\documentclass[ a4paper, 10pt]{amsart}
\usepackage{microtype, xfrac, enumerate, xspace}
\usepackage[latin1]{inputenc}
\usepackage[T1]{fontenc}
\usepackage{lmodern}
\usepackage[english]{babel}
\usepackage{amsmath}
\usepackage{amssymb}
\usepackage{mathtools}
\usepackage{mathrsfs}
\usepackage{xypic}
\usepackage{bookmark}
\usepackage{hyperref} 
\hypersetup{
	colorlinks=true,
	linkcolor=blue,
	filecolor=magenta,      
	urlcolor=cyan,
}
\usepackage[alphabetic]{amsrefs}
\usepackage{graphicx}
\usepackage{tikz}
\usepackage{tikz-cd}
\usetikzlibrary{matrix,arrows,decorations.pathmorphing}
\usepackage[all,cmtip]{xy}

\newtheorem*{main-thm-no-num}{Main Theorem}
\newtheorem*{thm-no-num}{Theorem}
\newtheorem*{df-no-num}{Definition}
\newtheorem{thm}{Theorem} [section]
\newtheorem{prop}[thm]{Proposition} 
\newtheorem{lm}[thm]{Lemma} 
\newtheorem{cor}[thm]{Corollary} 
\theoremstyle{remark}
\newtheorem{rmk}[thm]{Remark}

\theoremstyle{definition} 
\newtheorem {df}[thm]{Definition}

\newcommand{\Ccal}{\mathcal{C}}
\newcommand{\Dcal}{\mathcal{D}}
\newcommand{\Ga}{\mathbb{G}_{\rm a}}
\newcommand{\Gm}{\mathbb{G}_{\rm m}}

\newcommand{\Lcal}{\mathcal{L}}
\newcommand{\Mcal}{\mathcal{M}}
\newcommand{\Ncal}{\mathcal{N}}
\newcommand{\Ocal}{\mathcal{O}}
\newcommand{\PP}{\mathbb{P}}
\newcommand{\Pcal}{\mathcal{P}}
\newcommand{\Rcal}{\mathcal{R}}
\newcommand{\Scal}{\mathcal{S}}
\newcommand{\Vcal}{\mathcal{V}}
\newcommand{\Wcal}{\mathcal{W}}
\newcommand{\Zcal}{\mathcal{Z}}
\newcommand{\ZZ}{\mathbb{Z}}

\DeclareMathOperator{\B}{B}
\DeclareMathOperator{\C}{C}
\DeclareMathOperator{\CH}{CH}
\DeclareMathOperator{\GL}{GL}
\DeclareMathOperator{\id}{id}
\DeclareMathOperator{\im}{im}
\DeclareMathOperator{\pr}{pr}
\DeclareMathOperator{\Spec}{Spec}
\DeclareMathOperator{\Sym}{Sym}

\newenvironment{enumeratea}
{\begin{enumerate}[\upshape (a)]}
{\end{enumerate}}

\newenvironment{enumerate1}
{\begin{enumerate}[\upshape (1)]}
{\end{enumerate}}

\newenvironment{enumerateitem}
{\begin{enumerate}[$\bullet$]}
{\end{enumerate}}

\newtheorem*{namedtheorem}{\theoremname}
\newcommand{\theoremname}{testing}

\newtheorem{theorem}{Theorem}[section]

\newtheorem{proposition-definition}[theorem]
{Proposition-Definition}

\theoremstyle{definition}
\newtheorem{definition}[theorem]{Definition}
\newtheorem{notation}[theorem]{Notation}

\theoremstyle{remark}

\newcommand\nome{testing}
\newcommand\call[1]{\label{#1}\renewcommand\nome{#1}}
\newcommand\itemref[1]{\item\label{\nome;#1}}

\newcommand\refpart[2]{(\ref{#1;#2})}

 \newcommand\cB{\mathcal{B}}
 \renewcommand\cD{\mathcal{D}}
 \newcommand\cF{\mathcal{F}}
 
\newcommand\cI{\mathcal{I}} 
 
\newcommand\cM{\mathcal{M}} 
\newcommand\cO{\mathcal{O}} \newcommand\cP{\mathcal{P}}

 \newcommand\cV{\mathcal{V}}
 \newcommand\cX{\mathcal{X}}
\newcommand\cY{\mathcal{Y}} \newcommand\cZ{\mathcal{Z}}

\newcommand\GG{\mathbb{G}}

 \renewcommand\PP{\mathbb{P}}

 \renewcommand\ZZ{\mathbb{Z}}

\newcommand\rC{\mathrm{C}} 
 \newcommand\rF{\mathrm{F}}
 
\newcommand\rI{\mathrm{I}} 
 
 \newcommand\rN{\mathrm{N}}

\newcommand\rma{\mathrm{a}}

\newcommand\rmm{\mathrm{m}}

\newcommand\fM{\mathfrak{M}}

\newcommand\arr{\ifinner\to\else\longrightarrow\fi}

\newcommand\arrto{\ifinner\mapsto\else\longmapsto\fi}

\newcommand{\xarr}{\xrightarrow}

\renewcommand\H{\operatorname{H}}

\newcommand{\eqdef}{\mathrel{\smash{\overset{\mathrm{\scriptscriptstyle def}} =}}}

\def\displaytimes_#1{\mathrel{\mathop{\times}\limits_{#1}}}

\def\displayotimes_#1{\mathrel{\mathop{\bigotimes}\limits_{#1}}}

\newcommand\spec{\operatorname{Spec}}

\newcommand{\proj}{\operatorname{Proj}}

\newcommand\generate[1]{\langle #1 \rangle}

\newcommand{\underhom}{\mathop{\underline{\mathrm{Hom}}}\nolimits}

\newcommand{\underaut}{\mathop{\underline{\mathrm{Aut}}}\nolimits}

\newcommand{\undersym}{\mathop{\underline{\mathrm{Sym}}}\nolimits}

\newlength{\ignora}

\renewcommand{\setminus}{\smallsetminus}

\newcommand{\rest}[1]{{\mid_{#1}}}

\newcommand{\mmu}{\boldsymbol{\mu}}

\newcommand{\gm}{\GG_{\rmm}}
\newcommand{\gmp}[1]{\GG_{\rmm, #1}}

\newcommand{\PGL}{\mathrm{PGL}}
\newcommand{\ga}{\GG_{\rma}}

\DeclareFontFamily{U}{mathx}{\hyphenchar\font45}
\DeclareFontShape{U}{mathx}{m}{n}{
      <5> <6> <7> <8> <9> <10>
      <10.95> <12> <14.4> <17.28> <20.74> <24.88>
      mathx10
      }{}
\DeclareSymbolFont{mathx}{U}{mathx}{m}{n}
\DeclareFontSubstitution{U}{mathx}{m}{n}
\DeclareMathAccent{\widecheck}{0}{mathx}{"71}
\DeclareMathAccent{\wideparen}{0}{mathx}{"75}

\newcommand{\smallbullet}{\mathchoice{}{}%
{\raisebox{.14ex}{\scalebox{.6}\textbullet}}%
{\raisebox{.09ex}{\scalebox{.5}\textbullet}}%
}%

\renewcommand{\epsilon}{\varepsilon}

\newcommand{\cha}{\operatorname{char}}

\newcommand{\aff}[1][k]{(\operatorname{Aff}/#1)}

\newcommand{\ch}[1][*]{\operatorname{CH}^{#1}}

\begin{document}
\title[Polarized twisted conics and stable curves of genus two]%
{Polarized twisted conics and \\ moduli of stable curves of genus two}
\author[A. Di Lorenzo]{Andrea Di Lorenzo}
    \address[A. Di Lorenzo]{Aarhus University, Ny Munkegade 118, DK-8000 Aarhus C, Denmark}
    \email{andrea.dilorenzo@math.au.dk}
\author[A. Vistoli]{Angelo Vistoli}
    \address[A. Vistoli]{Scuola Normale Superiore, Piazza dei Cavalieri 7, 56126 Pisa, Italy}
    \email{angelo.vistoli@sns.it}
\thanks{Both authors have been partially supported by research funds from Scuola Normale Superiore}
\begin{abstract}
    In this paper we introduce the stack of polarized twisted conics and we use it to give a new point of view on $\overline{\Mcal}_2$. In particular, we present a new and independent approach to the computation of the integral Chow ring of $\overline{\Mcal}_2$, previously determined by Eric Larson.
\end{abstract}
\maketitle
\section*{Introduction}
Moduli stacks of curves play a prominent role in algebraic geometry. In particular, their rational Chow rings have been the subject of intensive research in the last forty years, since Mumford first investigated the subject in \cite{Mum}, and many explicit computations of rational Chow rings of moduli of curves have been produced over the years.

There is also a well defined notion of \emph{integral} Chow ring for these stacks, defined with the equivariant approach of \cite{EG}: this is more refined, but also much harder to compute. Almost all of the calculations that have been carried out (\cites{EG, EFRat, EF, FV, Dil, FVis, DLFV, Per}) have been for stacks of that can be presented as quotients $[(V\smallsetminus Y)/G]$, where $Y\subset V$ is a $G$-invariant closed subscheme of a $G$-representation $V$ for some linear algebraic group $G$ (we will refer to this as \emph{an easy presentation}). The single exception is the remarkable paper \cite{Lars}, in which Eric Larson computes the Chow ring of the stack $\overline{\cM}_{2}$ of stable curves, which is not known to have an easy presentation, by using the fact that the closed substack $\Delta_{1} \subseteq \overline{\cM}_{2}$ of curves with a separating node, and its complement $\overline{\cM}_{2} \setminus \Delta_{1}$, do have easy presentations. These presentations are completely different, and it is not know how to patch these together to give a useful presentation of $\overline{\cM}_{2}$ as a quotient stack. Larson overcomes this with a judicious use of the first higher equivariant Chow ring of $\overline{\cM}_{2} \setminus \Delta_{1}$. The result is the following.

\begin{thm-no-num}[E. Larson] The integral Chow ring of $\overline{\Mcal}_2$ is 
    \[
    \ZZ[\lambda_1,\lambda_2,\delta_1]/(2\delta_1^2+2\lambda_1\delta_1, \delta_1^3+\delta_1^2\lambda_1, 
    24\lambda_1^2-48\lambda_2,20\lambda_1\lambda_2-4\delta_1\lambda_2)
    \]
where the $\lambda_{i}$ are Chern classes of the Hodge bundle on $\overline{\cM}_{2}$, and $\delta_{1}$ is the class of $\Delta_{1} \subseteq \overline{\cM}_{2}$.
\end{thm-no-num}

In this paper we present a different approach to the study of the geometry of $\overline{\cM}_{2}$, and its Chow ring, which has some interesting features and is applicable to other cases of interest. Our goal is twofold:
\begin{enumerate}
    \item to introduce the stack $\Pcal$ of polarized twisted conics, and to show how the integral Chow ring of this stack can be computed \emph{without having a good quotient presentation}, by patching the Chow rings of two complementary substacks of $\Pcal$.
    \item to show how the stack $\overline{\Mcal}_2$ can be obtained as an open substack of a vector bundle $\Vcal$ over $\Pcal$, and how this leads to a presentation of $\CH^*(\overline{\Mcal}_2)$ after computing only the fundamental classes of some excised loci in $\Vcal$.
\end{enumerate}
All the main results of the paper can be summarized by the following.
\begin{main-thm-no-num}
Suppose that the characteristic of the base field is $\neq 2,3$. 
\begin{itemize}
    \item There exists a smooth algebraic stack $\Pcal$ whose objects are polarized twisted conics and whose integral Chow ring is
    \[ \CH^*(\Pcal)\simeq \ZZ[\lambda_1,\lambda_2,\delta_1,\eta]/(2\eta,\eta(\lambda_1\delta_1+\eta)). \]
    \item There exists a vector bundle $\Vcal$ over $\Pcal$ and closed substacks $\Dcal^3$, $\Zcal\subset\Vcal$ such that
    \[ \overline{\Mcal}_2\simeq \Vcal \smallsetminus \left(\Dcal^3\cup\Zcal\right).\]
    \item There exist another closed substack $\Dcal^{3,3}$ and a cycle $\zeta$ such that
    \[ \CH^*(\overline{\Mcal}_2)\simeq \CH^*(\Vcal)/([\Zcal],[\Dcal^3],\zeta,[\Dcal^{3,3}]). \]
    \item The integral Chow ring of $\overline{\Mcal}_2$ is isomorphic to
    \[ \ZZ[\lambda_1,\lambda_2,\delta_1]/(2\delta_1^2+2\lambda_1\delta_1, \delta_1^3+\delta_1^2\lambda_1, 24\lambda_1^2-48\lambda_2,20\lambda_1\lambda_2-4\delta_1\lambda_2). \]
\end{itemize}
\end{main-thm-no-num} 
In particular, we reprove Larson's theorem without using higher intersection theory, or test families. 

The stack $\overline{\cM}_{2} \setminus \Delta_{1}$ has an easy presentation of the type $[(V\smallsetminus Y)/\GL_{2}]$, where $V$ is a $7$-dimensional representation of $\GL_{2}$ (this is a straightforward extension of the presentation of $\cM_{2}$ used in \cite{VisM2}). This can be interpreted as follows: consider the vector bundle $\cV_{0} \ := [V/\GL_{2}]$ over the gerbe $\cB\GL_{2}$; then $\overline{\cM}_{2} \setminus \Delta_{1}$ is an open substack of $\cV_{0}$. 

Here we introduce a stack $\cP$, the stack of polarized twisted conics, that contains $\cP_{0} := \cB\GL_{2}$ as an open substack, whose complement $\cP_{1}$ is also the classifying stack of an algebraic group, an extension of $\cV_{0}$ to a vector bundle $\cV$ on $\cP$, and an open embedding $\overline{\cM}_{2} \subseteq \cV$. There is a natural rank~$2$ vector bundle $\Lambda$ on $\cP$ whose pullback to $\overline{\cM}_{2}$ is the Hodge bundle.

Since by homotopy invariance we have a canonical isomorphism of Chow rings $\CH^{*}(\cV) = \CH^{*}(\cP)$, we have that $\CH^{*}(\overline{\cM}_{2})$ is a quotient of $\CH^{*}(\cP)$. While $\cP$ does not have a good presentation, it Chow ring can be computed starting from the Chow rings of $\cP_{0}$ and $\cP_{1}$, and using the fact that the top Chern class of the normal bundle of $\cP_{1}$ in $\cP$ is not a zero divisor in the Chow ring $\ch(\cP_{1})$. This patching technique is at the basis of the Borel-Atiyah-Segal-Quillen localization theorem, and has been used by many authors working on equivariant cohomology, equivariant Chow rings and equivariant K-theory \cites{AtiEll, BreEq, Hsi, ChSk, KiQuo, GKMP, Bri, VV}. It has several other applications to calculations of Chow rings of stacks of curves, both integral and rational, which will be the subject of further work by the authors and by Michele Pernice. The result is that we have a presentation
   \[
   \ch(\cP) = \ZZ[\lambda_{1}, \lambda_{2}, \delta_{1}, \eta]/\bigl(2\eta, \eta(\lambda_{1}\delta_{1} + \eta)\bigr)\,,
   \]
where the $\lambda_{i}$ are the Chern classes of $\Lambda$, $\delta_{1}$ is the class of $\cP_{1}$, and $\eta$ is a certain codimension~$2$ class coming from the involution in $\cP_{1}$ that exchanges the two components of a reducible twisted curve.

This gives generators for $\ch(\overline{\cM}_{2})$. The next step is to find the relations; we find three natural closed substack $\cZ$, $\cD^{3}$ and $\cD^{3,3}$ of $\cV$, whose union is the complement of $\overline{\cM}_{2}$ in $\cV$, and we prove using geometric argument that the kernel of the surjection $\ch(\cP) = \ch(\cV) \arr \ch(\overline{\cM}_{2})$ is generated by the classes of $\cZ$, $\cD^{3}$ and $\cD^{3,3}$, and by an additional cycle $\zeta$. 

Finally, we compute the classes $[\cZ]$, $[\cD^{3}]$, $[\cD^{3,3}]$ and $\zeta$ in $\ch(\cP)$; this is rather technical, and in a key step we have to use the integral Grothendieck--Riemann--Roch theorem of Pappas \cite{Pap}, which is only known to hold in characteristic~$0$; fortunately, a small additional step allows to extend the result to characteristic different from $2$ and $3$.

\subsection*{Motivation for the introduction of twisted conics}
The quotient of a smooth curve $X$ of genus two by the hyperelliptic involution is a smooth conic $C$, and the quotient morphism is flat and finite of degree two, i.e. $X\to C$ is a flat double cover. This is a key ingredient to give a good presentation of $\Mcal_2$ as a quotient stack and eventually compute its integral Chow ring, as done by the second author in \cite{VisM2}.

This still works for the stack of stable curves of genus~$2$ that are not in $\Delta_{1}$, as these are all double covers of smooth conics. If $X$ is in $\Delta_{1}$, however, if $X \arr \overline{C}$ is the quotient by the hyperelliptic involution, then $C$ is singular conic, but the map $X \arr \overline{C}$ is not flat at the node.  On the other hand, the map above factors through the twisted conic $C$ obtained from $\overline{C}$ by adding a stacky structure of index $2$ at the node, in the sense of \cite{AVComp}, and the morphism $X\to C$ is actually a flat double cover. The standard description of double covers gives an embedding of $X$ in $L$, where $L$ is a line bundle of degree $1$ in the smooth case and bidegree $(\frac{1}{2},\frac{1}{2})$ in the singular case (the half degree makes sense thanks to the stacky structure).

The stack of smooth conics with a line bundle of degree~$1$ is the classifying stack $\cB\GL_{2}$; the stack $\cP$ which is the foundation of our construction is the stack of twisted conics, that is, either smooth conics, or twisted conics as described above.

\subsection*{Outline of the paper}
We summarize here the contents of the Sections. More detailed descriptions can be found at the beginning of the respective Sections.

In Section \ref{sec:twisted conics} we introduce twisted conics and polarized twisted conics. We study their local structure and their behavior in families. We define the stack $\Pcal$ of polarized twisted conics, and study its structure in detail. In particular, we show that it is a smooth quotient stack. We also introduce the vector bundle $\Lambda$ on $\cP$.

In Section \ref{sec:stable curves} we show how to obtain $\overline{\Mcal}_2$ as an open substack of a vector bundle $\Vcal$ over $\Pcal$, and show that $\Lambda$ pulls back to the Hodge bundle on $\overline{\cM}_{2}$.

In Section \ref{sec:chow P} we compute the integral Chow ring of $\Pcal$, the stack of polarized twisted conics, using the patching Lemma.

In Section \ref{sec:abstract} we give an abstract characterization of the integral Chow ring of $\overline{\Mcal}_2$: more precisely, building on our previous study of the Chow ring of $\Pcal$, we compute the generators of $\CH^*(\overline{\Mcal}_2)$ and we characterize the generators of the ideal of relations as fundamental classes of some excised loci in $\Vcal$.

In Section \ref{sec:concrete} we complete the computation of $\CH^*(\overline{\Mcal}_2)$: with the aid of some equivariant intersection theory, we produce explicit expressions for the relations, reproving Larson's result.

In the Appendix \ref{sec:app} we prove some results on deformation theory of nodal singularities. These results are used in the previous sections to study normal bundles of closed substacks.

\subsection*{Notation and assumptions}
We fix a base field $k$, of characteristic different from $2$ and $3$. All schemes and all morphisms considered will be defined over $k$; all stacks will be \'{e}tale stacks over the category $\aff$ of affine schemes over $k$. A product $X \times Y$ of schemes, or stacks, over $k$ will always be a fiber product over $\spec k$.

If $G$ is an algebraic group over $k$, we will denote by $\cB G$ the classifying stack of $G$-torsors.

Chow rings of quotient stacks are defined in \cite{EG}; we will not need the more general construction of \cite{Kre}.

\subsection*{Acknowledgments}
Part of this project has been developed while the authors were in residence at the Mathematical Sciences Research Institute in Berkeley, California, during the Spring 2019 semester. We acknowledge the support of the National Science Foundation under Grant No. DMS-1440140 and we wish to thank the staff at the MSRI for providing great working conditions.

We wish to thank Eric Larson for informing us of his result, and for the subsequent discussions.

\section{The stack of twisted conics}\label{sec:twisted conics}
In this Section, we define the stack $\Rcal$ of twisted conics (Definition \ref{def:stack R}) and the stack $\Pcal$ of polarized twisted conics (Definition \ref{def:stack P}). We prove that these stacks are algebraic (Propositions \ref{prop:R algebraic} and \ref{prop:Lcalzero iso BGLtwo}) and we describe a stratification of the stack $\Pcal$ of polarized twisted conics where each stratum is a classifying stack (Propositions \ref{prop:Lcalzero iso BGLtwo} and \ref{prop:Lcaluno iso BGamma}).
\subsection{Twisted conics}
If $S$ is a scheme, a \emph{conic} (short for reduced conic) over $S$ will be a morphism of schemes $\overline{C}\arr S$ which is proper, flat and finitely presented, whose geometric fibers are isomorphic to either $\PP^{1}$, or to the union of two copies of $\PP^{1}$ meeting transversally at a point.

 A \emph{twisted conic} over $S$ is an unmarked twisted curve $C \to S$, in the sense of \cite[Definition~2.1]{dan-olsson-vistoli-2}, whose geometric fibers are either $\PP^{1}$, or the union of two copies of $\PP^{1}$ meeting transversally at a point, with a stacky structure of index~$2$ at the node. Because of our hypotheses on $\cha k$, the stack $C$ is Deligne--Mumford. (Notice that our notation is different from that of \cite{dan-olsson-vistoli-2}, where the stack is denoted by $\Ccal$, and $C$ stands for its moduli space.) 

If $C \to S$ is a twisted conic, it is a tame Deligne--Mumford stack; hence formation of the moduli space $\overline{C} \arr S$ commutes with base change, and $\overline{C} \arr S$ is flat and finitely presented (see for example \cite[Corollary~3.3]{dan-olsson-vistoli1}); the fibers are clearly conics, hence $\overline{C}$ is a conic over $S$.

\begin{df}\label{def:stack R}
	We define $\Rcal \arr \aff$  as the stack of twisted conics, whose objects over an affine scheme $S$ are twisted conics $C \to S$. A morphism in $\Rcal$ from $C' \arr S'$ to $C \to S$ is an isomorphism class of $2$-cartesian diagrams
	\[
	\begin{tikzcd}
	C' \dar\rar &C\dar\\
	S' \rar & S
	\end{tikzcd}
	\]
	where the bottom arrow is a morphism of $k$-schemes (see \cite[Remark~2.5]{dan-olsson-vistoli-2}).
	We will call $\Rcal_{0} \subseteq \Rcal$ the open substack of smooth conics, whose objects are smooth proper families of curves of genus $0$, and $\Rcal_{1}$ the closed substack of singular curves.
\end{df}

If $C \to S$ is a proper flat finitely presented Deligne--Mumford stack, and the fiber over a geometric point $s_{0}\colon \spec\Omega \arr S$ is a twisted conic, there exists an open neighborhood $U$ of $s_{0}$ such that the restriction $C\rest U \arr U$ is a twisted conic (see \cite[Proposition~2.3]{dan-olsson-vistoli-2}). So $\Rcal$ is an open substack of the stack $\fM^{\mathrm{tw}}_{0,0}$ of unmarked twisted curves of genus~$0$.

\begin{prop}\label{prop:R algebraic}
The stack $\Rcal$ is algebraic, of finite type and smooth over $k$.
\end{prop}

\begin{proof}
By \cite[Theorem~A.6]{dan-olsson-vistoli-2}, the stack $\fM^{\mathrm{tw}}_{0,0}$ is algebraic, of finite type and smooth over $k$, so the same is true of $\Rcal$. The fact that $\Rcal$ is of finite type follows from the fact that there exists a unique twisted structure of index two on the union of two lines meeting transversally; hence the topological space $|\Rcal|$ has only two points, and $\Rcal$ is of finite type.
\end{proof}

\subsection{The local structure of twisted conics}\label{subsec:local-structure}
In order to better understand the structure of $\Rcal$ we need a local description of twisted conics. The following construction will be used extensively.

Let $R$ be a $k$-algebra, and $f \in R$. Consider the conic $D^{0}_{R,f} \eqdef \spec R[u,v]/(uv-f) \arr \spec R$; there is an action of $\mmu_{2}$ on $D^{0}_{R,f}$, by changing the signs of $u$ and $v$. The fixed point locus of this action is $F_{R,f} \eqdef \spec \bigl(R/(f)\bigr) = \spec R[u,v]/(f, u, v) \subseteq D^0_{R,f}$.

The ring of invariants $\bigl(R[u,v]/(uv-f)\bigr)^{\mmu_{2}}$ is easily checked to be $R[x,y]/(xy-f^{2})$, where $x = u^{2}$ and $y = v^{2}$. Call $C^{0}_{R,f} \arr S$ the stack theoretic quotient $[D_{R,f}/\mmu_{2}]$. The quotient $\Phi_{R,f} \eqdef [F_{R,f}/\mmu_{2}]$ is $\spec \bigl(R/(f)\bigr) \times \cB_{k}\mmu_{2}$.

Let $\overline{C}^{0}_{R,f} = \spec R[x,y]/(xy-f^{2})$ be its moduli space; there is a factorization $C^{0}_{R,f} \arr \overline{C}^{0}_{R,f} \arr \spec R$, and the projection $C^{0}_{R,f} \arr \overline{C}^{0}_{R,f}$ is an isomorphism outside of $\Phi_{R,f}$; in other words, we obtain an isomorphism $C^{0}_{R,f} \setminus \Phi_{R,f} \simeq  \overline{C}^{0}_{R,f} \setminus \spec \bigl(R/(f)\bigr)$, where $\spec \bigl(R/(f)\bigr)$ is embedded in $\overline{C}^{0}_{R,f}$ as $\spec R[x,y]/(f, x,y) \subseteq \spec R[x,y]/(xy - f)$.

Set
   \[
   \overline{C}_{R,f} \eqdef \proj R[x,y,z](xy - f^{2}z^{2}) \arr \spec R\,;
   \]
this is a conic over $\spec R$, containing $\overline{C}^{0}_{R,f}$ as an open subscheme, dense in every fiber. 
\begin{notation}\label{notation:CRf}
	We denote by $C_{R,f} \arr \spec R$ the stack obtained by gluing $C^{0}_{R,f}$ with $\overline{C}_{R,f} \setminus \spec \bigl(R/(f)\bigr)$ using the isomorphism $C^{0}_{R,f} \setminus \Phi_{R,f} \simeq \overline{C}^{0}_{R,f} \setminus \spec \bigl(R/(f)\bigr)$ above.
\end{notation} 
One easily checks that $C_{R,f} \arr \spec R$ is a twisted conic, with moduli space $C_{R,f} \arr \overline{C}_{R,f}$. The singular locus of the map $C_{R,f} \arr \spec R$, which coincides with the stacky locus of $C_{R,f}$, is $\Phi_{R,f} \subseteq C_{R,f}$.

We will also need the following. Consider the compactification
   \begin{equation}\label{eq:DRf}
   D_{R,f} \eqdef \proj R[u,v,w]/(uv - fw^{2})\,;
   \end{equation}
the action of $\mmu_{2}$ on $D^{0}_{R,f}$ extends to an action of $\mmu_{2}$ on $D_{R,f}$, which is free outside of $F_{R,f}$ and of the locus at infinity 
   \[
   E_{R,f} = \proj R[u,v](uv) = \spec R \sqcup \spec R \subseteq D_{R,f}\,.
   \]
The morphism $D^{0}_{R,f} \arr C^{0}_{R,f}$ extends to a flat morphism $D_{R,f} \arr C_{R,f}$. This is \'{e}tale outside of $E_{R,f}$; it is ramified of order $2$ along $E_{R,f}$. The composite $E_{R,f} \subseteq D_{R,f} \arr C_{R,f}$ is also a closed embedding, so we can (and will) also view $E_{R,f}$ as a subscheme of $C_{R,f}$.

The importance of this construction is revealed by the following.

\begin{prop}\label{prop:local-twisted-conic}
Let $C \to S$ be a twisted conic, $s\colon \spec\Omega \arr S$ a geometric point. Then $s$ has an \'{e}tale neighborhood $\spec\Omega \arr \spec R \arr S$, such that the pullback of $C$ to $\spec R$ is isomorphic to $C_{R,f}$ for some $f \in R$.
\end{prop}

\begin{proof}
The idea of the proof is to show that $C$ comes, \'{e}tale-locally, from a conic of the type $D_{R,f} \arr \spec R$, as above. The map $D_{R,f} \arr C_{R,f}$ is ramified of order $2$ along $E_{R,f} \subseteq C_{R,f}$, and \'{e}tale elsewhere. We will go in the inverse direction, by constructing, \'{e}tale-locally on $S$, a relative Cartier divisor $E \subseteq C$, showing that \'{e}tale-locally it has an double covering $D \arr C$ ramified along $E$, and then showing that Zariski-locally $D \simeq D_{R,f}$ for some $k$-algebra $R$ and some $f \in R$.

If the geometric fiber $C_{s}$ is smooth, the statement is clear; so, assume that $C_{s}$ is singular. Furthermore, since the stack $\Rcal$ is of finite type over $k$, we can also assume that $S$ is of finite type over $k$.

Pick two smooth points $p$ and $q$ in $C_{s}(\Omega)$ in the smooth locus, and two different components of the smooth locus of $C_{s}$. By passing to an \'{e}tale neighborhood we can assume that there are two disjoint sections $S \arr C$ landing in the smooth locus of $C \arr S$ whose inverse images in $C_{s}$ are $p$ and $q$ respectively. Call $E \subseteq C$ the union of these sections; we want to show that \'{e}tale-locally $C$ has a double cover $D \arr C$, where $D \arr S$ is a conic, ramified on $E$. This can be interpreted as an \'{e}tale-cover of the stack $\widetilde{C} \arr S$ obtained by taking a square root of the Cartier divisor $E\subseteq C$, in the sense of \cite[Appendix~B]{dan-tom-angelo2008}.

Set $C_{0} \eqdef C_{\Omega,0}$ (we are adopting here Notation \ref{notation:CRf}); then $C_{0}$ is the union of two copies of $\PP^{1}$ meeting transversally at a point, with a stacky structure of index~$2$ at the node. Set $D_{0} = D_{\Omega,0}$ (same notation as in (\ref{eq:DRf})); the ramified double cover $D_{0} \arr C$ gives an \'{e}tale double cover $D_{0} \arr \widetilde{C}_{0} \simeq \widetilde{C}_{s}$. Denote by $s_{0} \in S$ the image of $s\colon \spec\Omega \arr S$; this double cover will be defined over an intermediate extension $k(s_{0}) \subseteq k' \subseteq \Omega$ that is finite and separable over $k(s_{0})$. By passing to an \'{e}tale neighborhood, we can assume that $k' = k(s_{0})$. 

If $S' \subseteq S$ is an infinitesimal neighborhood, the \'{e}tale double cover $D_{s_{0}} \arr \widetilde{C}_{s_{0}}$ extends uniquely to an \'{e}tale double cover $D_{S'} \arr C_{S'}$. From Grothendieck's existence theorem for Deligne--Mumford stacks \cite[Proposition~2.1]{olsson-starr-quot} we obtain an \'{e}tale double cover $D_{\widehat{\cO}_{S,s_{0}}} \arr \widetilde{C}_{\widehat{\cO}_{S,s_{0}}}$. From Artin approximation we deduce that after passing to an \'{e}tale neighborhood there exists an \'{e}tale double cover $D \arr \widetilde{C}$, as claimed. Clearly, $D \arr S$ is a conic; the induced morphism $D \arr C$ is a double cover, and the embedding $E \subseteq C$ lifts to an embedding $E \subseteq D$. There is an action of $\mmu_{2}$ on $D$, acting trivially on $E$, such that $\widetilde{C} = [D/\mmu_{2}]$; thus, $\widetilde{C}$, and therefore $C$, are uniquely determined by $D$ and the action of $\mmu_{2}$. We can assume that $S = \spec R$ is affine. It is enough to prove that there is an element $f \in R$, and a $\mmu_{2}$ equivariant isomorphism of $S$-schemes $D \simeq D_{R,f}$.

Call $\pi\colon D \arr S$ the projection, and notice that the invertible sheaf $\cO_{D}(E)$ is ample on $D$, hence we get an embedding $D \subseteq \PP\bigl(\H^{0}(D,\cO_{D}(E))\bigr) = \PP^{2}_{R}$. We have an exact sequence
   \[
   0 \arr \H^{0}(D,\cO_{D}) \arr \H^{0}\bigl(D,\cO_{D}(E)\bigr) \arr \H^{0}(D,\cO_{E}) \arr 0\,;
   \]
since $\H^{0}(D,\cO_{D}) = R$ and $\H^{0}(D,\cO_{E}) = R^{2}$, we can split the sequence, and get a basis $x$, $y$, $z$ for $\H^{0}\bigl(D,\cO_{D}(E)\bigr)$ (here $z$ is supposed to be a generator of $\H^{0}(D,\cO_{D})$). After passing to a neighborhood of $s_{0}$ we can assume that $\cO_{\PP^{2}_{R}}(D)$ is isomorphic to $\cO_{\PP^{2}_{R}}(2)$. So $D$ is defined by a quadratic equation, which can not contain neither $x^{2}$ nor $y^{2}$, so it is of the form
   \[
   xy + (ax+by)z + c z^{2} = (x+bz)(y+az) + (c - ab)z^{2}\,.
   \]
Thus we have an isomorphism $D \simeq D_{R,f}$, where $f := c - ab \in R$. This isomorphism sends $E$ into $E_{R,f}$.

It is easy to check that $D_{R,f}$ has a unique involution that acts like the identity on $E$, and does not act trivially on any component of a geometric fiber; hence the isomorphism $D \simeq D_{R,f}$ is $\mmu_{2}$-equivariant. This finishes the proof.
\end{proof}

If $C \to S$ is a twisted conic, call $\Sigma_{C} \subseteq C$ the scheme-theoretic singular locus of the projection $C \to S$, defined, as usual, by the first Fitting ideal $\rF_{1}\Omega_{C/S} \subseteq \cO_{C}$ (since $C$ is a Deligne--Mumford stack, there is no difficulty in defining the sheaf of K\"{a}hler differentials $\Omega_{C/S}$). Formation of $\Sigma_{C}$ commutes with base change. The image of $\Sigma_{C}$ in $S$ is the set of points $s\in S$, such that the fiber of $C$ over $s$ is singular.

\begin{prop}\call{prop:local-description-Sigma}\hfil
\begin{enumerate1}

\itemref 1 If $R$ is a $k$-algebra and $f \in R$, then $\Sigma_{C_{R,f}} = \Phi_{R,f} \subseteq C_{R,f}$.

\itemref 2 The composite $\Sigma_{C} \subseteq C \to S$ is a gerbe banded by $\mmu_{2}$ over a closed subscheme $S_{1}$ of $S$. 

\end{enumerate1}
\end{prop}

\begin{proof}
Part \refpart{prop:local-description-Sigma}{1}: since the projection $D_{R,f} \arr C^{0}_{R,f}$ is \'{e}tale, this boils down to the statement that the subscheme of $D_{R,f}$ defined by the first Fitting is $F_{R,f}$, which is a simple calculation. Part~\refpart{prop:local-description-Sigma}{2} follows from Proposition~\ref{prop:local-twisted-conic}, and from the fact that $\Phi_{R,f}$ equals $\spec \bigl(R/(f)\bigr) \times \cB\mmu_{2}$.
\end{proof}

\begin{definition}\label{def:stack C}
	We will denote by $\Ccal \arr \Rcal$ the universal curve; in other words, an object of $\Ccal(S)$ is a pair $(C \to S, p)$, where $p\colon S \arr C$ is a section of $C \to S$. Since formation of $\Sigma_{C/S}$ for a twisted conic $C \to S$ commutes with base change, we also obtain a closed substack $\Sigma_{\Ccal} \subseteq \Ccal$, which is the singular locus of $\Ccal$. 
\end{definition}

Every smooth family is locally trivial, that is, locally a product with $\PP^{1}$, so $\Rcal_{0}$ is isomorphic to the classifying stack $\cB\PGL_{2}$. The restriction $\Ccal_{0}$ of $\Ccal$ to $\Rcal_{0}$ is the quotient $[\PP^{1}/\PGL_{2}]$.
The complement $\Rcal_{1}$ of $\Rcal_{0}$ in $\Rcal$ has a natural scheme structure.

\begin{prop}
The scheme-theoretic image $\Rcal_{1} \subseteq \Rcal_{0}$ of $\Sigma_{\Ccal} \subseteq \Ccal$ is a smooth effective divisor, whose complement is $\Rcal_{0}$. Furthermore, $\Sigma_{\Ccal} \arr \Rcal_{1}$ is a gerbe banded by $\mmu_{2}$.
\end{prop}

\begin{proof}
This follows immediately from Proposition~\ref{prop:local-description-Sigma}.
\end{proof}

Proposition~\ref{prop:local-description-Sigma} also allows to describe $\Rcal_{1}$ more explicitly. Notice that a morphism $S \arr \Rcal$ corresponding to a twisted conic $C \to S$ factors through $\Rcal_{1}$ if and only if $S_{1} = S$. Consider the twisted conic $C_{k,0}$ (see Notation \ref{notation:CRf}); then Proposition~\ref{prop:local-twisted-conic} implies that \'{e}tale-locally over $S$ the twisted conic $C$ is isomorphic to $S \times C_{k,0}$. So, in particular, $\Rcal_{1}$ is an \'{e}tale gerbe. Since $C_{k,0} \arr \spec k$ is in $\Rcal_{1}(\spec k)$ we have that it is a classifying stack $\cB\underaut_{k}C_{k,0}$.

\subsection{Invertible sheaves on twisted conics}\label{subsec:invertible-on-twisted}

Let $C \simeq C_{\Omega,0}$ be a twisted conic over an algebraically closed field $\Omega$, and let $L$ be a invertible sheaf on $C$. Then $L$ has a well-defined degree: see \cite[\S~7.2]{dan-tom-angelo2008} for a discussion.

If $C$ is singular, then it is the intersection of two copies of $\PP^{1}$, with a twisted structure of index~$2$ at the node; call $C_{1}$ and $C_{2}$ its irreducible components; these are copies of $\PP^{1}$ with a stacky point of index-$2$ ($1$-marked smooth twisted curves of genus~$0$, in the terminology of \cite{dan-olsson-vistoli-2}). An invertible sheaf on $C_{i}$ has a degree, which is a multiple of $1/2$. The degree $\deg L$ is the sum $\deg(L\rest{C_{1}}) + \deg(L\rest{C_{2}})$; the pair $\bigl(\deg(L\rest{C_{1}}), \deg(L\rest{C_{2}})\bigr)$ is the multidegree of $L$ (this is only well defined up to ordering). An invertible sheaf on $C$ is called \emph{balanced} if $\deg(L\rest{C_{1}}) = \deg(L\rest{C_{2}})$.

If $C \to S$ is a twisted conic and $L$ is an invertible sheaf on $C$, the degree of $L$ restricted to the geometric fibers of $C \to S$ is a locally constant function on $S$. The invertible sheaf $L$ is called \emph{balanced} if its restriction to any of the singular fibers is balanced.

Set $\Sigma = \Sigma_{C} \simeq \cB\mmu_{2}$. The inverse image $\Sigma_{i}$ of $\Sigma$ in  $C_{i}$ is isomorphic to $\Sigma$.

\begin{lm}\label{lem:integer-degrees}
The half-integers $\deg(L\rest{C_{1}})$ and $\deg(L\rest{C_{2}})$ are integral if an only if the restrictions of $L$ to $\Sigma$ is isomorphic to $\cO_{\Sigma}$.
\end{lm}

\begin{proof}
This reduces to the statement that in invertible sheaf on $C_{i}$ has integer degree if and only if its restriction to $\Sigma_{i}$ is trivial, which is straightforward.
\end{proof}

\begin{prop}\label{prop:locally-isomorphic-balanced}
Let $\phi\colon C \arr S$ be a twisted conic, $L_{1}$ and $L_{2}$ balanced invertible sheaves on $C$ of the same degree~$d$. Then there exists an \'{e}tale covering $S' \arr S$ such that the pullbacks of $L_{1}$ and $L_{2}$ to $C_{S'}$ are isomorphic.
\end{prop}

\begin{proof}
Consider the invertible sheaf $M = L_{1}\otimes L_{2}^{\vee}$, which is balanced of degree~$0$.

Assume that $S$ is the spectrum of an algebraically closed field. Then from Lemma~\ref{lem:integer-degrees} we see that $M$ is trivial along the gerbe $\Sigma \subseteq C$. If $\pi\colon C \arr \overline{C}$ is the moduli space, it is a standard fact that $\pi_{*}M$ is an invertible sheaf on $\overline{C}$, and $M = \pi^{*}\pi_{*}M$; hence it is enough to show that $\pi_{*}M$ is trivial on $\overline{C}$. But $\pi_{*}M$ has degree~$0$ on $\overline{C}$, and bidegree~$(0,0)$ on the singular fibers, so this is true.

Also, $\H^{1}(C, M) = \H^{1}(\overline{C}, \pi_{*}M) = 0$; so we can apply the cohomology and base change theorem for algebraic stacks \cite{hall-base-change} and deduce that $\phi_{*}M$ is an invertible sheaf on $S$, and $M = \phi^{*}\phi_{*}M$. The result follows immediately.
\end{proof}

\subsection{Polarized twisted conics}

\begin{df}
A \emph{polarization} on a twisted conic $C \to S$ is a balanced invertible sheaf on $C$ of constant degree~$1$.

A \emph{polarized twisted conic} over a scheme $S$ is a pair $(C \to S, L)$, where $C \to S$ is a twisted conic and $L$ is a polarization on $C \to S$.
\end{df}

Locally, every twisted conic has a polarization.

\begin{prop}\label{prop:locally-polarized}
Let $C \arr S$ be a twisted conic. Then there exists an \'{e}tale covering $S' \arr S$ such that the pullback $C_{S'} \arr S'$ has a polarization.
\end{prop}

\begin{proof}
By Proposition~\ref{prop:local-twisted-conic} we can assume that $S = \spec R$, $f \in R$, and $C = C_{R,f}$. We have seen that there exists a flat double cover $\pi\colon D_{R,f} \arr C_{R,f}$ that is \'{e}tale outside of the union $E_{R,f}\subseteq C_{R,f}$ of two disjoint sections $\spec R \arr C_{R,f}$, landing outside the singular locus of $C_{R,f}$. Also, $E_{R,f}$ meets each singular fibers in two point lying in different components. Furthermore, $\pi$ is ramified of order~$2$ along $E_{R,f}$.

We have a decomposition $\pi_{*}\cO_{D_{R,f}} = \cO_{C_{R,f}} \oplus M$, where $M \subseteq \pi_{*}\cO_{D_{R,f}}$ is the subsheaf of trace~$0$ elements. From the standard theory of double covers we know that $M$ is an invertible sheaf; if $L = M^{\vee}$, then $L^{\otimes 2}$ has a section vanishing exactly on $E_{R,f}$. It follows that the restriction of $L^{\otimes 2}$ to the fibers of $C \arr S$ has degree $2$, and bidegree $(1,1)$ in the singular fibers. So $L$ is a polarization.
\end{proof}

\begin{df}\label{def:stack P}
	We will call $\cP$ the stack of polarized twisted conics. If $S$ is a scheme, an object of $\cP(S)$ is a polarized twisted conic on $S$; a morphism from $(C' \arr S', L')$ to $(C \to S, L)$ is a pair consisting of a morphism in $\Rcal$
	\[
	\begin{tikzcd}
	C' \dar\rar{R,f} &C\dar\\
	S' \rar & S
	\end{tikzcd}
	\]
	and an isomorphism $L' \simeq f^{*}L$ of sheaves of $\cO_{C'}$-modules. We will call $\cP_{0}$ the inverse image of $\Rcal_{0}$, and $\cP_{1}$ the inverse image of $\Rcal_{1}$.
\end{df}
There is an obvious morphism $\cP \arr \Rcal$, sending $(C \to S, L)$ into $C \to S$.

If $\cY \arr \cX$ is a morphism of \'{e}tale stacks over $\aff$, we say that $\cY$ is a $\gm$-gerbe over $\cX$ if the following conditions hold. 
\begin{enumerate1}
\item For every affine scheme $S$ and every morphism $S \arr \cX$, the projection $S \times_{\cX}\cY \arr S$ has the structure of a gerbe banded by $\gm$.

\item If $S' \arr S$ is a morphism of schemes, the induced morphism $S' \times_{\cX}\cY \arr S\times_{\cX}\cY$ is a morphism of gerbes banded by $\gm$.
\end{enumerate1}

More explicitly, this translates into the following conditions.

\begin{enumeratea}\call{conditions-gm-gerbe}

\itemref1 If $\eta$ is an object of $\cY$ over an affine scheme $S$, and $\xi$ is its image in $\cX$, we have a given isomorphism of \'{e}tale sheaves of groups $\gmp{S} \simeq \ker\bigl(\underaut_{S}\eta \arr \underaut_{S}X\bigr)$. These are required to be compatible with base change under morphisms $S' \arr S$

\itemref2 If $\xi$ is an object of $\cX(S)$ for some affine scheme $S$, there exists an \'{e}tale covering $S' \arr S$ such that the pullback $\xi_{S'}$ has a lifting to $\cY(S')$.

\itemref3 If $\xi$ is an object of $\cX(S)$ for some affine scheme $S$, and $\eta_{1}$ and $\eta_{2}$ are liftings of $\xi$ to $\cY(S)$, then after passing to an \'{e}tale covering of $S$, the two lifting are isomorphic.

\end{enumeratea}

If $\cY \arr \cX$ is a $\gm$-gerbe and $\cX$ is an algebraic stack, it is easy to see that $\cY$ is also an algebraic stack, and $\cY \arr \cX$ is smooth.

\begin{prop}
The morphism $\cP \arr \Rcal$ makes $\cP$ into a $\gm$-gerbe over $\Rcal$.
\end{prop}

\begin{proof}
We need to check conditions \refpart{conditions-gm-gerbe}{1}, \refpart{conditions-gm-gerbe}{2} and \refpart{conditions-gm-gerbe}{3} above.

Part~\refpart{conditions-gm-gerbe}{1} follows immediately from the fact that if $C \to S$ is a twisted conic, the natural map $\cO^{*}(S) \arr \cO^{*}(C)$ is an isomorphism. Part~\refpart{conditions-gm-gerbe}{3} follows from Proposition~\ref{prop:locally-isomorphic-balanced}, and \refpart{conditions-gm-gerbe}{2} is the content of Proposition~\ref{prop:locally-polarized}. \end{proof}
From this it follows that the stack $\cP$ is algebraic, and smooth over $k$.
Since both $\Rcal_{0}$ and $\Rcal_{1}$ are gerbes, it follows that $\cP_{0}$ and $\cP_{1}$ are also gerbes. In particular, $\cP$ is of finite type. 

The structure of $\cP_{0}$ is easy to determine: $\bigl(\PP^{1}, \cO(1)\bigr)$ is an object of $\cP_{0}(k)$, hence $\cP_{0}$ is the classifying stack $\cB\underaut_{k}\bigl(\PP^{1}, \cO(1)\bigr)$ of automorphism group scheme of the pair $(\PP^{1}, \cO(1))$; but it is well known that $\underaut_{k}\bigl(\PP^{1}, \cO(1)\bigr) = \GL_{2}$, so $\cP_{0} = \cB\GL_{2}$.  So we proved the following.

\begin{prop}\label{prop:Lcalzero iso BGLtwo}
The stack $\cP$ is algebraic, and smooth over $k$. The open substack $\cP_{0}$ and the closed substack $\cP_{1}$ are algebraic gerbes over $k$. Moreover we have $\Pcal_0\simeq \cB\GL_2$.
\end{prop}
In \ref{subsec:automorphism-polarized} we will describe $\cP_{1}$; this is quite a bit more elaborate.

\subsection{Automorphism groups of polarized twisted conics}\label{subsec:automorphism-polarized}

First of all, set $C = C_{k,0}$, $D = D_{k,0}$, $E = E_{k,0}$ (see Notation \ref{notation:CRf} and (\ref{eq:DRf})). That is, the curve $D$ is a union of two copies of $\PP^1$, glued in $0$. The divisor $E$ is the divisor that on each component is given by the point at infinity, and $C$ is the twisted conic obtained from $[D/\mmu_2]$ after rigidifying the two points at the infinity.

We can view $E$ as a subscheme of $D$, of $C$, or of the moduli space $\overline{C}$ of $C$. Then $\Sigma = \Sigma_{C} \subseteq C$ is the gerbe at the node, $C^{0} \eqdef C \setminus E$, $D^{0} \eqdef D \setminus E$, $C^{0} \eqdef C \setminus E$. 
The group $\mmu_{2}$ acts on $D$, leaving $E$ fixed; the map $D \arr C$ ramified of order $2$ along $E$, and \'{e}tale elsewhere.% We have $D^{0} = \spec k[x,y]/(xy)$; the action of $\mmu_{2}$ changes the signs of $x$ and $y$; also, by construction we have $C^{0} = [D^{0}/\mmu_{2}]$.

Let us start by determining the structure of the automorphism group scheme $\underaut_{k}C$. This is a little subtle; an automorphism of $C$ is not determined by its restriction to $C \setminus \Sigma$. In other words, if $\overline{C}$ is the moduli space of $C$, the natural homomorphism $\underaut_{k}C \arr \underaut_{k}\overline{C}$ is not injective. It is surjective, but its kernel is isomorphic to $\mmu_{2}$ (see \cite[\S 7.1]{dan-alessio-vistoli03}); the non-trivial automorphism in the kernel is known as a \emph{ghost automorphism}.

The group scheme $\underaut_{k}C$ has a subgroup $\underaut_{k}(C,E)$ consisting of automorphisms sending $E$ to itself. Let us start by working out the structure of $\underaut_{k}(C,E)$. First of all, notice that the connected component of the identity of $\underaut_{k}(D,E)$ is isomorphic to $\gm^{2}$; the action of $\gm^{2}$ on $D$ is determined by fixing an ordering $D_{1}$, $D_{2}$ of the two irreducible components of $D$, as follows.

\begin{enumerateitem}

\item The first component $\gm \subseteq \gm^{2}$ acts on $D_{1}$, fixing the points $D_{1} \cap D_{2}$ and $D_{1}\cap E$, in such a way that the differential of the action at the point $D_{1}\cap E$ is given by multiplication on the tangent space of $D_{1}$ at $D_{1}\cap E$.

\item The action of the first component of $\gm^{2}$ is trivial on $D_{2}$. 

\item The same holds for the second component of $\gm^{2}$, exchanging $D_{1}$ and $D_{2}$.

\end{enumerateitem}

Choose an isomorphism $\phi\colon D_{1} \simeq D_{2}$ that sends $D_{1} \cap D_{2}$ to itself and exchanges $D_{1} \cap E$ with $D_{2} \cap E$. From $\phi$ one gets an involution of $(D,E)$, which switches $D_{1}$ and $D_{2}$, and generates a cyclic subgroup $\rC_{2} \subseteq \underaut_{k}(D,E)$. One checks easily that $\underaut_{k}(D,E)$ is a semidirect product $\gm^{2}\rtimes \rC_{2}$, where the generator of $\rC_{2}$ switches the two copies of $\gm$.

Consider the subgroup $\mmu_{2}^{2} \subseteq \gm^{2} \subseteq \underaut_{k}(D,E)$, and the diagonal embedding $\mmu_{2} \subseteq \mmu_{2}^{2} \subseteq \underaut_{k}(D,E)$; this embeds $\mmu_{2}$ into $\underaut_{k}(D,E)$ as a characteristic subgroup; the subgroup $\underaut_{k}(D,E) \subseteq \underaut_{k}D$ is the normalizer of $\mmu_{2}$. The morphism $D \arr C$, which is ramified of order $2$ along $E$, is $\mmu_{2}$-invariant, therefore it induces a morphism $[D/\mmu_{2}] \arr C$; this is an isomorphism outside of $E$. The stack quotient $[D/\mmu_{2}]$ is a twisted curve $[D/\mmu_{2}] \arr C$, with a gerbe of index $2$ at that node, and two markings of index $2$, whose union is $[E/\mmu_{2}] = E \times\cB_{k}\mmu_{2}$. The map induces an isomorphism $\underaut_{k}\bigl([D/\mmu_{2}], [E/\mmu_{2}]\bigr) \simeq \underaut_{k}(C,E)$. On the other hand $\mmu_{2} \subseteq \underaut_{k}(D,E)$ is a central subgroup, hence there exists a induced homomorphism
   \[
   \underaut_{k}(D,E) \arr \underaut_{k}\bigl([D/\mmu_{2}], [E/\mmu_{2}]\bigr) = \underaut_{k}(C,E)\,.
   \]
Furthermore, $[D/\mmu_{2}]$ has a unique connected \'{e}tale cover of degree $2$, that is $D \arr [D/\mmu_{2}]$; this implies that the homomorphism above is surjective, with kernel $\mmu_{2}$.

We claim that $\underaut_{k}C$ is a semidirect product $\ga^{2} \rtimes \underaut_{k}(C,E)$ for an action of $\underaut_{k}(C,E)$ on $\ga^{2}$, which we will not need to write down. For this it is enough to produce an embedding $\ga^{2} \subseteq \underaut_{k}C$ as a normal subgroup, with the property that for each affine scheme $S$ every element of $\underaut_{k}C(S)$ can be written in a unique was as $\alpha\beta$, where $\alpha \in \ga^{2}(S)$ and $\beta \in \underaut_{k}(C,E)(S)$.

First of all, if $\overline{C}_{1}$ and $\overline{C}_{2}$ are the irreducible components of $\overline{C}$, fix isomorphisms $\overline{C}_{i} \simeq \PP^{1}$, in such a way that the point of intersection of the two components corresponds to the point at infinity. This induces actions of $\ga$ on $\overline{C}_{1}$ and $\overline{C}_{2}$ by translation, which give an action of $\ga^{2}$ on $\overline{C}$. This identifies $\ga^{2}$ with a normal subgroup of $\underaut_{k}\overline{C}$.

Now consider the inverse image of $\ga^{2}$ in $\underaut_{k}C$; since for every extension
   \[
   1 \arr \mmu_{2} \arr E \arr \ga^{2} \arr 1
   \]
there is a unique splitting $\ga^{2} \arr E$, we get a unique lifting of $\ga^{2}\subseteq \underaut_{k}\overline{C}$ to a normal subgroup $\ga^{2} \subseteq \underaut_{k}C$. 

If $S$ is an affine scheme and $\phi$ is an automorphism of $C\times S$ over $S$, then it is easy to see that there exists a unique element $\alpha$ of $\ga(S)$ carrying $E \times S$ onto $\alpha(E \times S)$; this implies what we want, that is, that $\underaut_{k}C$ is semidirect product $\ga^{2} \rtimes \underaut_{k}(C,E) = \ga^{2} \rtimes\bigl( (\gm^{2}/\mmu_{2})\rtimes\rC_{2}\bigr)$. Of course it is possible to write the action of $\gm^{2}/\mmu_{2}$ on $\ga^{2}$ explicitly, but this will not be necessary for what follows. Thus, $\Rcal_{1}$ is described as the classifying stack of the group $\ga^{2} \rtimes\bigl( (\gm^{2}/\mmu_{2})\rtimes\rC_{2}\bigr)$.

Now let us produce an object in $\cP_{1}$ over $\spec k$. Call $q$ the point $D_{1} \cap D_{2}$; let $M$ be the invertible sheaf on $D$ obtained by gluing at $q$ the sheaves $\cO_{D_{1}}(q)$ and $\cO_{D_{2}}(q)$ along the fiber on $q$, using the differential of $\phi\colon D_{1} \simeq D_{2}$ at the origin. Then $M$ has degree $1$ on each component. The action of the group $\mmu_{2}$ on $D$ lifts to an action on $M$; the action of the generator of $\mmu_{2}$ on the fiber over $q$ is by changing sign. Hence $M$ descends to an invertible sheaf $\overline{M}$ on $[D/\mmu_{2}]$. Since the action of $\mmu_{2}$ on the fibers of $M$ over the two points of $E$ is trivial, the pushforward of $\overline{M}$ to $C$ is an invertible sheaf $L$ that pulls back to $\overline{M}$. This invertible sheaf $L$ has degree $1/2$ on each components, so $(C,L)$ is an object of $\cP_{1}(k)$. 

The automorphism group of $L$ as a sheaf of $\cO_{C}$-modules is $\gm$; we have a central short exact sequence
   \[
   1 \arr \gm \arr \underaut_{k}(C, L) \arr \underaut C \arr 1
   \]
(the surjectivity on the right follows from Proposition~\ref{prop:locally-isomorphic-balanced}). Every extension of $\ga^{2}$ by $\gm$ splits uniquely, so if we denote by $\Gamma$ the inverse image of $(\gm^{2}/\mmu_{2})\rtimes \rC_{2} = \underaut_{k}(C,E)$ in $\underaut_{k}(C, L)$, we have $\underaut_{k}(C, L) = \ga^{1}\rtimes\Gamma$. We need to work out the structure of $\Gamma$. 

Call $\overline{\Gamma}$ the inverse image in $\underaut_{k}(D, M)$ of $\gm^{2} \rtimes \rC_{2} = \underaut_{k}(D, E) \subseteq \underaut D$. We have a central extension
   \[
   1 \arr \gm \arr \overline{\Gamma} \arr \gm^{2} \rtimes \rC_{2} \arr 1
   \]
where $\gm$ is the group of automorphisms of $M$ as a sheaf of $\cO_{D}$-modules; every element of $\overline{\Gamma}$ acts linearly on the fiber of $M$ over $q$; this gives a homomorphism $\overline{\Gamma} \arr \gm$ which splits the embedding $\gm \subseteq \overline{\Gamma}$. Hence we have an isomorphism	$\overline{\Gamma} \simeq \bigl(\gm^{2} \rtimes \rC_{2}\bigr)\times\gm$; so $\overline{\Gamma}$ is a semidirect product $\gm^{3}\rtimes \rC_{2}$, where the action of $\rC_{2}$ switches the first and the second component of $\gm^{3}$. We have a surjective homomorphism $\overline{\Gamma} \arr \Gamma$, whose kernel is $\mmu_{2} \subseteq \overline{\Gamma} \subseteq \underaut_{k}(D,M)$, acting on $(D,M)$ in the way described above. Under the isomorphism $\Gamma \simeq \gm^{3}\rtimes \rC_{2}$ this corresponds to the diagonal embedding $\mmu_{2} \subseteq \mmu_{2}^{3} \subseteq \gm^{3}$. We have proved the following.

\begin{prop}\label{prop:Lcaluno iso BGamma}
The stack $\cP_{1}$ is the classifying stack of a group of the form $\ga^{2}\rtimes \Gamma$, where $\Gamma$ is a group scheme of the form $(\gm^{3}/\mmu_{2})\rtimes \rC_{2}$, where $\mmu_{2} \subseteq \gm^{3}$ is embedded diagonally, and $\rC_{2}$ acts by switching the first and the second component of $\gm^{3}$.
\end{prop}

\subsection{Two vector bundles on the stack of polarized twisted conics}\label{subsec:two-vector-bundles}

Let $\pi\colon\Ccal_\Pcal\to\Pcal$ be the universal conic over $\Pcal$, obtained by pulling back the universal conic $\Ccal$ over the stack $\Rcal$ of twisted conics (Definition \ref{def:stack C}) along the forgetful morphism $\Pcal\to\Rcal$. By construction, the universal conic $\Ccal_\Pcal$ has a universal polarization, which we denote $\Lcal$. 

Let $\Lambda$ be the coherent sheaf defined as $\pi_*\Lcal$, so that we have
\[ \Lambda(\pi\colon C\to S, L)= \pi_*L.  \]
Applying the cohomology and base change theorem (\cite{hall-base-change}) we see that $\Lambda$ is a locally free sheaf of rank $2$. 

Let us describe the fibres of $\Lambda$ over the two geometric points of $\Pcal$. The unique (up to isomorphism) geometric point of $\Pcal_0$ is the smooth conic $\PP^1$, with polarization $\Ocal(1)$. Therefore
\[ \Lambda(\PP^1,\Ocal(1))= \H^0(\PP^1,\Ocal(1)), \]
As already stated in Proposition \ref{prop:Lcalzero iso BGLtwo}, we have that $\underaut_{k}(\PP^1,\Ocal(1))$ is isomorphic to $\GL_2\simeq\underaut_{k}(\H^0(\PP^1,\Ocal(1))$. In particular we have proved the following.
\begin{lm}\label{lm:P0 faithful}
    The action of $\underaut_k(\PP^1,\Ocal(1))$ on $\Lambda(\PP^1,\Ocal(1))$ is faithful.
\end{lm}
Consider the morphism $f\colon\spec k\to \Rcal_1$ given by the twisted conic $C$: recall that $C$ is obtained from $D=\PP^1\cup_0\PP^1$ by taking the quotient with respect to the $\mmu_2$ action on $D$ and rigidifying the points at infinity. Let $L$ be the polarization on $C$ obtained by descending along $D\to C$ the line bundle $M:=\Ocal(0)\cup_0\Ocal(0)$. Recall that we have
\begin{align*}
    \underaut_k(D,M)=\ga^2\rtimes \overline{\Gamma}=\ga^2\rtimes \left(\gm^3\rtimes C_2\right), \\
    \underaut_k(C,L)=\ga^2\rtimes \Gamma=\ga^2\rtimes \left((\gm/\mmu_2)^3\rtimes C_2\right).
\end{align*}
Let $X_0\oplus (0)$, $(0)\oplus X_0$ and $X_1\oplus X_1$ be the standard generators of $M=\Ocal(1)\cup_0 \Ocal(1)$.
If $(u,v,t)$ is an element of $\gm^3\subset\overline{\Gamma}$, then by construction it acts on $\H^0(D,M)$ as follows:
\begin{align*}
    (u,v,t)\cdot X_0\oplus (0) &= tu^{-1} X_0\oplus (0), \\
    (u,v,t)\cdot (0)\oplus X_0 &= (0)\oplus tv^{-1} X_0, \\
    (u,v,t) \cdot X_1\oplus X_1 &= t (X_1\oplus X_1).
\end{align*}
Consider the diagonal embedding $\mmu_2\subset\gm^3$. The vector space $\H^0(C,L)=\H^0(D,M)^{\mmu_2}$ is generated by the first two elements above. We deduce the following:
\begin{lm} \label{lm:P1 almost faithful}
    The kernel of the morphism $\ga^2\rtimes \Gamma\arr \GL_2$ induced by the representation $H^0(C,L)$ is isomorphic to $\gm/\mmu_2$, embedded via $t^2\mapsto (1,1,t^2)$. In particular, the action of $\underaut_k(C,L)$ on $\Lambda(C,L)$ is not faithful.
\end{lm}
In the next few lines, we will use some results on deformations of nodal singularities which are well known to the experts. Nonetheless, we included in Appendix \ref{sec:app} a more detailed discussion of this topic, as we could not find any reference in the literature.

From Proposition \ref{prop:lci} it follows that the closed subscheme $\cP_{1} \subseteq \cP$ is a smooth divisor; hence it has a normal bundle, which is an invertible sheaf on $\cP_{1}$. By Proposition~\ref{prop:Lcaluno iso BGamma} this corresponds to a character of $\ga^{2}\rtimes \Gamma$, which comes from a character of $\Gamma$. We will need to identify this character. Notice that $\cP_{1}$ is the pullback of $\Rcal_{1} \subseteq \Rcal$ under the natural map $\cP \arr \Rcal$, which is smooth; hence the normal bundle comes from $\Rcal$. In other words, the character of $\Gamma$ we are interested with comes from the character of $(\gm^{2}/\mmu_{2})\rtimes \rC_{2}$ corresponding to the normal bundle of $\Rcal_{1}$ in $\Rcal$.

Consider the morphism $f\colon\spec k\to \Rcal_1$ given by the twisted conic $C$: recall that $C$ is obtained from $D=\PP^1\cup_0\PP^1$ by taking the quotient with respect to the $\mmu_2$ action on $D$ and rigidifying the points at infinity. Let $g:D\to C$ be the corresponding morphism. Then from Proposition \ref{prop:normal-in-the-base} we have
\[ f^*\Ncal_{\Pcal_1}=g_*((\Ocal_{\PP^1}(0)\otimes\Ocal_{\PP^1}(0))|_0), \]
where $\Ncal_{\Pcal_1}$ is the normal bundle of $\Pcal_1$.

The group $\underaut_k(D)=\ga^{2}\rtimes(\gm^{2}\rtimes \rC_{2})$ acts on the vector space $(\Ocal_{\PP^1}(0)\otimes\Ocal_{\PP^1}(0))|_0$ by multiplication by $uv$, where $u$ and $v$ are the coordinates of the two $\gm$ components. This character is $\mmu_2$ invariant, hence descends to a character of $(\gm^{2}/\mmu_{2})\rtimes \rC_{2}$. This character is therefore the one associated to the normal bundle of $\Rcal_1$ in $\Rcal$ and, consequently, to $\Pcal_1$ in $\Pcal$. We have proved the following.
\begin{lm}\label{lm:normal bundle P1}
    Consider the description of the group $\Gamma$ as a quotient given by
    \[\gm^{3}\rtimes \rC_{2}\longrightarrow\Gamma,\quad (u,v,t,\sigma)\longmapsto (ut^{-1},vt^{-1},uv,\sigma), \]
    and let $\chi$ be the character of $\ga^{2}\rtimes\Gamma$ given by $uv$. If $V_{\chi}$ is the associated rank one representation, then we have
    \[\Ncal_{\Pcal_1}\simeq [V_{\chi}/\ga^{2}\rtimes\Gamma]\]
\end{lm}

\subsection{Quotient structure}

A quotient stack is defined as in \cite{EHKV}, as a quotient $[X/G]$ where $X$ is an algebraic space of finite type over $k$, and $G$ is a linear group scheme over $k$ acting on $X$.

\begin{prop}
The stack $\cP$ is a quotient stack.
\end{prop}

\begin{proof}
We will use the criterion of \cite[Lemma~2.12]{EHKV}; so it is enough to produce a locally free sheaf $E$ on $\cP$ such that the action of the stabilizers of geometric points $p\colon \spec\overline{k} \arr \cP$ on the fiber $E(p)$ is faithful.

The stack $\cP$ has, up to isomorphism, two geometric points, corresponding to the smooth polarized conic $(\PP^1,\Ocal(1))$ and to the twisted polarized conic $(C,L)$ that we introduced before. Consider the vector bundle on $\Pcal$ defined as
\[ E:=\Lambda\oplus\Ocal(\Pcal_1). \]
By Lemma \ref{lm:P0 faithful}, the action of $\underaut_k(\PP^1,\Ocal(1))$ on $\Lambda(\PP^1,\Ocal(1))$ is faithful, hence also the action on $E(\PP^1,\Ocal(1))$ must be faithful.

By Lemma \ref{lm:P1 almost faithful}, the stabilizer of the action of $\underaut_k(C,L)$ on $\Lambda(C,L)$ is a copy of $\gm/\mmu_2$. On the other hand, we know from Lemma \ref{lm:normal bundle P1} that the aforementioned subgroup acts faithfully on $\Ncal_{\Pcal_1}=\Ocal(\Pcal_1)|_{\Pcal_1}$. Therefore, the action of $\underaut_k(C,L)$ on $E$ is faithful. Putting all together, we deduce that $\Pcal$ is a quotient stack.
\end{proof}

\section{The stack of stable curves of genus two via twisted conics}\label{sec:stable curves}
In this Section we show how to obtain the moduli stack $\overline{\Mcal}_2$ as an open substack of a vector bundle $\Vcal$ over the stack $\Pcal$ of polarized twisted conics (Theorem \ref{prop:M2bar iso}). For this, after introducing the already mentioned vector bundle (Definition \ref{def:stack V}), we investigate some closed substacks of $\Vcal$ (Definitions \ref{def:stack D3}, \ref{def:stack D33} and \ref{def:stack Z}), whose union will turn out to be the complement of the moduli stack $\overline{\Mcal}_2$ in $\Vcal$.

\subsection{A vector bundle over the stack of polarized twisted conics}

Recall from Definition \ref{def:stack P} that the objects of the stack $\Pcal$ of polarized twisted conics are pairs $(C\to S,L)$, where $C\to S$ is a twisted conic and $L$ is a balanced line bundle over $C$ of degree $1$, i.e. it has degree $1$ on the smooth fibers and bidegree $(\frac{1}{2},\frac{1}{2})$ on the singular ones.

A straightforward application of the cohomology and base change theorem (\cite{hall-base-change}) shows that the coherent sheaf \begin{equation}\label{eq:stack V}\Vcal:=\pi_*\left(\Lcal^{\otimes 2}\otimes\omega_{\Ccal_\Pcal/\Pcal}^{\otimes (-2)}\right)\end{equation}
is locally free of rank $7$.
\begin{df}\label{def:stack V}
We define the stack $\Vcal$ as the vector bundle over $\Pcal$ introduced in (\ref{eq:stack V}). The objects of $\Vcal$ are polarized twisted conics $(C\to S,L)$ together with a global section $s$ of the line bundle $\omega_{C/S}^{\otimes(-2)}\otimes L^{\otimes 2}$.
The restriction of $\Vcal$ to $\Pcal_0$ (resp. $\Pcal_1$) is denoted by $\Vcal_0$ (resp. $\Vcal_1$).
\end{df}
The stack $\Vcal$ is a smooth algebraic stack over $k$, because so is $\Pcal$ by Proposition \ref{prop:Lcalzero iso BGLtwo}.

\subsection{Substacks of sections with special vanishing properties}\label{subsec:closed substacks}
We want to define a few closed substacks of the stack $\Vcal$ by imposing some vanishing conditions on the sections. By excising these closed substacks we will obtain an open substack of $\Vcal$ isomorphic to the stack of stable curves of genus two. For convenience, set
\[ \cF:= \Lcal ^{\otimes 2} \otimes \omega_{\Ccal_\Pcal/\Pcal}^{\otimes (-2)}.\]
Consider the cartesian square
\[ \xymatrix{
\Ccal_{\Pcal}\times_{\Pcal}\Ccal_\Pcal \ar[r]^{\pr_2} \ar[d]^{\pr_1} & \Ccal_\Pcal \ar[d]^{\pi} \\
\Ccal_\Pcal \ar[r]^{\pi} & \Pcal }\]
and let $\Delta_{\Ccal/\Pcal}$ be the closed substack of $\Ccal_{\Pcal}\times_{\Pcal}\Ccal_\Pcal $ given by the diagonal embedding. We have an injective morphism of sheaves
\[ \pr_2^*\cF\otimes \cI_{\Delta_{\Ccal/\Pcal}}^3 \hookrightarrow \pr_2^*\cF \] 
where the term on the left is the line bundle $\cF$ tensored with the third power of the ideal sheaf of $\Delta_{\Ccal/\Pcal}$. After taking the pushforward along the first projection we get
\[ \pr_{1*}\left( \pr_2^*\cF\otimes \cI_{\Delta_{\Ccal/\Pcal}}^3 \right)\hookrightarrow \pr_{1*}\pr_2^*\cF. \]
The right term is isomorphic to $\pi^*\pi_*\cF=\pi^*\Vcal$. Away from the singular locus of $\Ccal_\Pcal$, the injective morphism above is an embedding of locally free sheaves. In other terms, we have defined a vector subbundle $\Wcal$ of the restriction of $\Vcal\times_{\Pcal}\Ccal_{\Pcal}$ to the smooth locus of the universal conic. The stack $\Wcal$ will play a role in \ref{subsec:relations from W}.

By construction, the objects of this vector subbundle are of the form $(C\to S,L,s,p)$ where $s$ is a global section of $L^{\otimes 2}\otimes\omega_{C/S}^{\otimes (-2)}$, $p\colon S\to C$ is a section of $C\to S$ supported in the smooth locus (hence a Cartier divisor), and $s$ vanishes with order at least $3$ along the image of $p$. The closure of this substack in the vector bundle $\Vcal\times_{\Pcal}\Ccal_{\Pcal}$ coincides with the closed substack that can be defined as the relative spectrum
\begin{equation}\label{eq:D3} \underline{{\rm Spec}}_{\Ocal_{\pi^*\Vcal}}\left( \underline{{\rm Sym}} \left( \pr_2^*\cF\otimes \cI_{\Delta_{\Ccal/\Pcal}}^3 \right)^{\vee} \right).  \end{equation}

\begin{df}\label{def:stack D3}
We define the closed substack $\Dcal^3\subset \Vcal$ as the scheme-theoretic image (\cite[\href{https://stacks.math.columbia.edu/tag/0CMH}{Tag 0CMH}]{stacks-project}) along the proper morphism $\Vcal\times_{\Pcal}\Ccal_{\Pcal}\to\Vcal$ of the closed substack (\ref{eq:D3}).
\end{df}
The generic point of $\Dcal^3$ corresponds to a polarized twisted conic $(C,L)$ together with a section $s$ of $L^{\otimes 2}\otimes\omega_{C}^{\otimes (-2)}$ whose vanishing locus has a triple point.

In the same way, we can define the closed substack $\Dcal^{3,3}\subset\Vcal$ parametrizing sections with two triple roots and their degenerations. Indeed, consider the fibered product
\[ \Ccal_{\Pcal}\times_{\Pcal} \Ccal_{\Pcal} \times_{\Pcal} \Ccal_{\Pcal}. \]
Let $\pr_{13}^*\Delta_{\Ccal/\Pcal}$ be the pullback of the diagonal along the projection on first and third factor, and let $\pr_{23}^*\Delta_{\Ccal/\Pcal}$ be the pullback of the diagonal along the projection on the second and third factor. Then over $\Ccal_{\Pcal}\times_{\Pcal}\Ccal_{\Pcal}$ we have an injective morphism of sheaves
\begin{equation}\label{eq:D33} \pr_{12*}\left( \cI_{\pr_{13}^*\Delta_{\Ccal/\Pcal}\cup \pr_{23}^*\Delta_{\Ccal/\Pcal}} \otimes \pr_3^*\cF\right)\hookrightarrow  \pr_3^*\cF \end{equation}
The relative spectrum of the symmetric algebra generated by the sheaf on the left defines a closed substack of $\Vcal\times_{\Pcal} \Ccal_{\Pcal}\times_{\Pcal}\Ccal_{\Pcal}$. Away from the singular locus, the points of this closed substack are polarized twisted conics with two markings $p_1$ and $p_2$ and a section of $L^{\otimes 2}\otimes\omega_{C}^{\otimes (-2)}$ having a triple root both in $p_1$ and $p_2$.
\begin{df}\label{def:stack D33}
We define the closed substack $\Dcal^{3,3}\subset\Vcal$ as the scheme-theoretic image (\cite[\href{https://stacks.math.columbia.edu/tag/0CMH}{Tag 0CMH}]{stacks-project}) along the proper morphism $\Vcal\times_{\Pcal} \Ccal_{\Pcal}\times_{\Pcal}\Ccal_{\Pcal}\to  \Vcal$ of the relative spectrum of the symmetric algebra generated by the sheaf on the left of (\ref{eq:D33})
\end{df}

The last closed substack of $\Vcal$ that we want to consider is, roughly speaking, the one parametrizing those sections that vanish on the singular locus of singular polarized twisted conics. The construction is similar to the ones already discussed: consider the injective morphism of coherent sheaves
\[ \cI_{\Sigma_{\Ccal_\Pcal}} \hookrightarrow \Ocal_{\Ccal_\Pcal} \]
where on the left we have the ideal sheaf of the singular locus of $\Ccal_\Pcal$ (Definition \ref{def:stack C}). After taking the tensor product with $\cF$, the pushforward along $\pi\colon\Ccal_\Pcal \to \Pcal$ gives us
\[ \pi_*\left(\cI_{\Sigma_{\Ccal_\Pcal}}\otimes\cF \right) \hookrightarrow \pi_*\cF=\Vcal \]
This induces a closed embedding as follows
\[ \underline{{\rm Spec}}_{\Ocal_{\Pcal}}\left(\underline{{\rm Sym}}~\pi_*\left(\cI_{\Sigma_{\Ccal_\Pcal}}\otimes\cF\right)^\vee \right)\hookrightarrow \underline{{\rm Spec}}_{\Ocal_\Pcal}\left(\underline{{\rm Sym}} \left(\Vcal^\vee\right) \right), \]
where the term on the right is the total space of $\Vcal$.
\begin{df}\label{def:stack Z}
We define the closed substack $\Zcal$ inside $\Vcal$ as the left hand side of the embedding above.
\end{df}
Observe that by construction an object in $\Zcal$ is of the form $(C\to S,L,s)$ where the polarized twisted conic $(C\to S,L)$ is singular and the global section $s$ of $L^{\otimes 2}\otimes \omega_{C/S}^{\otimes (-2)}$ vanishes along the singular locus $\Sigma_C$.

\subsection{The stack of stable curves of genus two}
We are ready to prove our first main result.
\begin{thm}\label{prop:M2bar iso}
   Suppose that the base field $k$ has characteristic $\neq 2$. Then we have
   \[ \overline{\Mcal}_2\simeq\Vcal\smallsetminus\left(\Dcal^3\cup \Zcal\right)\]
\end{thm}
\begin{proof}
Recall from Definition \ref{def:stack V} that the objects of $\Vcal$ are of the form $(C\to S,L,s)$, where $C\to S$ is a twisted conic with a polarization $L$ and $s$ is a global section of $L^{\otimes 2}\otimes\omega_{C/S}^{\otimes (-2)}$. 

If an object belongs to the closed substack $\Dcal^3$ (Definition \ref{def:stack D3}) then it has the following property: on each fiber of $C\to S$, the vanishing locus of the section must have at least a triple point, or it must vanish on the singular locus. In this second case, there are restraints on the vanishing order of the section that follow from how $\Dcal^3$ is defined, but we don't need to be more precise about these conditions.

If an object $(C\to S,L,s)$ belongs to the closed substack $\Zcal$ (Definition \ref{def:stack Z}), then the section $s$ must vanish on the singular locus of the twisted conic. Therefore, the objects of the open substack $\Vcal\smallsetminus\left(\Dcal^3\cup \Zcal\right)$ are of the form $(C\to S,L,s)$ where on each fiber of $C\to S$ the section $s$ can have at most a double root, and it cannot vanish on the singular locus.

First we are going to construct a morphism from $\Vcal\smallsetminus\left(\Dcal^3\cup \Zcal\right)$ to $\overline{\Mcal}_2$ . Given an object $(C\to S,L,s)$ in $\Vcal\smallsetminus\left(\Dcal^3\cup \Zcal\right)$, we will show how to construct a family of stable curves of genus $2$ over $S$. The strategy is basically the same adopted in \cite{ArsVis}.

Consider the $\Ocal_C$-algebra given by
\[ \Ocal_C \oplus \left(L^\vee\otimes\omega_{C/S}\right) \]
where the multiplicative structure is determined by the two morphisms
\[ \Ocal_C\otimes \left(L^\vee\otimes\omega_{C/S}\right) \longrightarrow L^\vee\otimes\omega_{C/S},\quad \left(L^\vee\otimes\omega_{C/S}\right)^{\otimes 2}\longrightarrow \Ocal_C,\]
where the map on the left is the evaluation at the section $s$. Define
\[ C':=\underline{{\rm Spec}}_{\Ocal_C}\left( \Ocal_C \oplus \left(L^\vee\otimes\omega_{C/S}\right) \right)\longrightarrow C, \]
and let $\overline{C}$ be the coarse moduli space of $C$.
There is a well defined coaction
\[ \Ocal_C \oplus \left(L^\vee\otimes\omega_{C/S}\right) \longrightarrow \left( \Ocal_C \oplus \left(L^\vee\otimes\omega_{C/S}\right) \right) \otimes k[t]/(t^2-1)\]
that sends an element $(a,b)$ to $a\otimes 1 + b\otimes t$. Therefore, we have an action of $\mmu_2$ on $C'$ which turns it into a ramified cover of degree $2$ over $C$. The ramification locus is by construction the vanishing locus of $s$: as the section $s$ does not vanish on the singular locus $\Sigma_C$ of $C$, this implies that $C'\to C$ is \'{e}tale in a neighbourhood of $\Sigma_C$. As \'{e}tale locally $\Sigma_C\simeq \Sigma_{\overline{C}}\times \cB\mmu_2$, this implies that $\Sigma_{C'}$ is a scheme and it is fixed by the action of $\mmu_2$.

Putting all together, we deduce the following: first, the stack $C'$ is actually a scheme. Second, the morphism $C'\to \overline{C}$ is a double cover ramified along the union of the vanishing locus of $s$ with the singular locus. By the Hurwitz formula, this implies that $C'\to S$ is a family of curves of genus two. Remember that the vanishing locus of $s$ can have at most double points: by direct inspection, this implies that the fibers of $C'\to S$ can have at most nodal singularities. In other terms, we have proved that $C'\to S$ is a family of stable curves of genus two.

It is easy to see that the association $(C\to S,L,s)\mapsto (C'\to S)$ is actually a functor and hence it defines a morphism
\[ \Vcal\smallsetminus\left( \Dcal^3\cup\Zcal \right)\longrightarrow \overline{\Mcal}_2.  \]
To construct an inverse we argue as follows: given a family of stable curves of genus two $C'\to S$ we can take the quotient stack $[C'/\mmu_2]$, whose coarse moduli space $\overline{C}\to S$ is a family of conics. The quotient stack $[C'/\mmu_2]$ is a scheme away from the image of the locus fixed by the involution, which is the union of two divisors: one is the singular locus $\Sigma_{C'}$ of $C'$, and the other is its complement $D$ in the fixed locus. Observe that the image of the fixed locus is by construction a gerbe over the ramification locus in $\overline{C}$ banded by $\mmu_2$.

We can rigidify the gerbe supported on the image of $D$. Let us denote $C\to S$ the resulting rigidified stack, which is a twisted conic. The resulting map $f\colon C'\to C$ is representable, flat and finite of degree $2$. The standard theory of cyclic covers tells us that $f_*\Ocal_{C'}\simeq \Ocal_{C}\oplus \cF$ for some line bundle $\cF$ over $C$, and that
\[ C'\simeq \underline{{\rm Spec}}_{\Ocal_C}\left( \Ocal_C\oplus\cF \right), \]
with the multiplicative structure of the algebra $\Ocal_C\oplus\cF$ determined by a section $s$ of $\cF^{\otimes (-2)}$. Moreover, the ramification locus of $C'\to C$ must coincide with the vanishing locus of $s$.

This  implies that $\cF$ must have degree $-3$ on the smooth fibers of $C\to S$ and bidegree $(-\frac{3}{2},-\frac{3}{2})$ on the singular ones. Therefore, the pair $(C\to S,\cF^\vee \otimes \omega_{C/S})$ is a polarized twisted conic, and $(C\to S,\cF^\vee \otimes \omega_{C/S},s)$ is an object of $\Vcal$.

The fact that the fibers of $C'\to S$ are stable easily implies that the ramification locus of $C'\to C$ can have at most double points, and we already checked that it cannot contain the singular locus. As the ramification locus coincides with the vanishing locus of $s$, we deduce that $(C\to S,\cF^\vee \otimes \omega_{C/S},s)$ actually belongs to $\Vcal\smallsetminus \left(\Dcal^3\cup \Zcal\right)$.
The functoriality of the association 
\[ (C'\to S)\longmapsto (C\to S,\cF^\vee \otimes \omega_{C/S},s) \]
is straightforward to verify. We have in this way constructed a morphism
\[ \Vcal\smallsetminus\left(\Dcal^3\cup\Zcal \right) \longrightarrow \overline{\Mcal}_2. \]
A direct inspection shows that the two morphisms of stacks that we have defined are one the inverse of the other, thus concluding the proof.
\end{proof}
In particular, we have constructed a morphism $\overline{\Mcal}_2\to \Pcal$. Observe that the preimage of $\Pcal_1$ in $\overline{\Mcal}_2$ is the divisor $\Delta_1$ of stable curves of genus two with at least one separating node. The generic point of $\Delta_1$ correspond to a curve formed by two smooth curves of genus one glued in one point.

Recall that the \emph{Hodge bundle} on $\overline{\cM}_{2}$ associates with each family $\rho\colon X \arr S$ in $\overline{\cM}_{2}$ the sheaf $\rho_{*}\omega_{\rho}$ on $S$; it is a locally free sheaf of rank~$2$.

\begin{prop}
The pullback of $\Lambda$ to $\overline{\cM}_{2}$ along the projection $\overline{\cM}_{2} \arr \cP$ is canonically isomorphic to the Hodge bundle.
\end{prop}

\begin{proof}
The argument is essentially contained in \cite{VisM2}. Consider a family $\rho\colon X \arr S$ in $\overline{\cM}_{2}$; denote by $(C,L)$ the associated twisted conic. We have a tautological factorization $X \xarr{f} C \xarr{\pi} S$; we need to produce a canonical functorial isomorphism $\rho_{*}\omega_{\rho} \simeq \pi_{*}L$.

By Grothendieck duality, we have a canonical isomorphism 
\[f_{*}\omega_{\rho} \simeq \underhom_{\cO_{C}}(f_{*}\cO_{X}, \omega_{\pi}).\]
On the other hand, by construction we have a tautological isomorphism $f_{*}\cO_{X} \simeq \cO_{C} \oplus (L^{\vee} \otimes_{{\cO}_{C}}\omega_{\pi})$; so we obtain a canonical isomorphism $f_{*}\omega_{\rho} \simeq \omega_{\pi} \oplus L$, hence, by applying $\pi_{*}$, a canonical isomorphism $\rho_{*}\omega_{\rho} \simeq \pi_{*}\omega_{\pi} \oplus \pi_{*}L$. Since $\pi_{*}\omega_{\pi} = 0$, this concludes the proof.
\end{proof}

\section{The integral Chow ring of the stack of polarized twisted conics}\label{sec:chow P}
This Section is devoted to the study of the integral Chow ring of the stack of polarized twisted conics. The main result is Theorem \ref{prop:CH Lcal}, where the integral Chow ring of $\Pcal$ is computed in terms of generators and relations.
\subsection{Computation of $\CH^*(\Pcal_1)$}\label{subsec:chow P1}
Let $(\pi\colon\Ccal_\Pcal\to\Pcal,\Lcal)$ be the universal polarized twisted conic over $\Pcal$. We denote by $\Lambda_i$ the restriction of $\Lambda$ to $\Pcal_i$, for $i=0,1$.
\begin{prop}\label{prop:CH Lcaluno}
   We have
    \[ \CH^*(\Pcal_1)\simeq \ZZ[\lambda_1,\lambda_2,\delta_1,\xi]/(2\xi,\xi(\xi+\lambda_1)) \]
    where the cycles $\lambda_j$ are the Chern classes of $\Lambda_1$, the cycle $\delta_1$ is the first Chern class of the normal line bundle of $\Pcal_1$ in $\Pcal$, and the cycle $\xi$ is pullback of the generator of $\CH^*(\cB\C_2)$ along the morphism $\Pcal_1\to \cB\C_2$ induced by the isomorphism of Proposition \ref{prop:Lcaluno iso BGamma}.
\end{prop}
\begin{proof}
In the following, we adopt the same notation used in the proof of Proposition \ref{prop:Lcaluno iso BGamma}.
In particular, recall that $\overline{\Gamma}:=\Gm^3\rtimes \C_2$, with $\C_2$ switching the first and the second factor, and that $\Gamma:=\left(\Gm/\mmu_2\right)^3\rtimes \C_2$.

Observe that $\overline{\Gamma}$ is abstractly isomorphic to $\Gamma$, where the isomorphism is induced by the homomorphism $\overline{\Gamma}\to \overline{\Gamma}$ defined as
\[((u,v,t,\sigma) \longmapsto (ut^{-1},vt^{-1}, uv,\sigma)\]
whose kernel is exactly $\mmu_{2}$.

Moreover, by \cite{MolVis}*{Lemma 2.3} we have 
\[\CH^*(\cB\left(\ga^2\rtimes\Gamma\right))\simeq\CH^*(\cB\Gamma)\simeq \CH^*(\cB(\Gm^{ 2}\rtimes \C_2)\times\cB\Gm)\simeq \CH^*(\cB(\Gm^{ 2}\rtimes \C_2))[\delta_1]\]
where $\delta_1$ is the first Chern class of the rank one representation determined by the projection on the third copy of $\Gm$, i.e. the one where $\Gamma$ acts by multiplication by $uv$.

By Lemma \ref{lm:normal bundle P1} we know that the normal bundle of $\Pcal_1$ in $\Pcal$ is the quotient by $\Ga^2\rtimes\Gamma$ of the rank one representation associated to the character $uv$. Therefore, the cycle $\delta_1$ coincides with the first Chern class of the normal bundle.

We have to compute $\CH^*(\cB((\Gm^{2}\rtimes \C_2))$, for which we apply \cite{VisPGLp}*{Proposition 6.4.(c)}: this tells us that there exists a ring homomorphism
\[ \phi\colon\CH^*(\cB\GL_2)\longrightarrow\CH^*(\cB(\Gm^{2}\rtimes \C_2))\]
that makes $\CH^*(\cB((\Gm^{2}\rtimes \C_2))$ into a $\CH^*(\cB\GL_2)$-algebra.
%consider the homomorphism $\Gm^{ 2}\ltimes \C_2\to \GL_2$ where $\Gm^{ 2}$ is identified with the maximal torus of diagonal matrices and $\C_2$ with the subgroup of transpositions.

%This induces a morphism of classifying stacks $B(\Gm^{ 2}\ltimes \C_2)\to B\GL_2$, and the pullback along this map

Moreover, by \cite{VisPGLp}*{Proposition 6.4.(c)} we have:
\[ \CH^*(\cB((\Gm^{2}\rtimes \C_2)))\simeq \CH^*(\cB\GL_2)[\xi]/I\simeq \ZZ[\phi(c_1),\phi(c_2),\xi]/I \]
where $\xi$ is the first Chern class of the sign representation induced by the projection on $\C_2$.

The ideal of relations $I$ is generated by $2\xi$ and by $\xi\phi(\eta)$ for each $\eta$ in the kernel of the pullback morphism
\begin{equation}\label{eq:pullback}\CH^*(\cB\GL_2)\simeq\ZZ[c_1,c_2]\longrightarrow\CH^*(\cB\mmu_{2})\simeq \ZZ[\xi]/(2\xi) \end{equation}
induced by the homomorphism $\mmu_{2}\hookrightarrow \GL_2$. The pullback of the standard representation of $\GL_2$ along $\cB\mmu_{2}\to \cB\GL_2$ is the direct sum of two copies of the sign representation, hence $c_1\mapsto 0$ and $c_2\mapsto \xi^2$. Henceforth, the kernel of this morphism of rings (\ref{eq:pullback}) is generated by $c_1$, thus $I=(2\xi,\xi\phi(c_1))$.

As in the proof of Proposition \ref{prop:Lcaluno iso BGamma} let $D$ be the union of two copies of $\PP^1$, glued in $0$, and let $C$ be the twisted conic obtained from $[D/\mmu_2]$ after rigidifying the two points at the infinity. Let $L$ be the standard polarization on $C$, and set $M:=f^*L$, the pullback of $L$ along $f\colon D\to C$.
By definition of $\Lambda$ we have:
\[ \Lambda(C,L)=\H^0(C,L)=\H^0(D,M)^{\mmu_{2}} \]
where $\mmu_{2}\subset \overline{\Gamma}$ is the usual kernel of $\overline{\Gamma}\to\Gamma$. We can describe $\H^0(D,M)$ as the subspace of $\H^0(\PP^1,\Ocal(0))^{\oplus 2}$ generated by those elements $s\oplus t$ such that $s(0)=t(0)$.

If $X_0$ and $X_1$ denote the standard global sections of $\Ocal(0)$, then $\mmu_{2}$ acts by sending $X_1$ to $-X_1$ and leaving $X_0$ unchanged. Therefore a base for $\H^0(D,M)^{\mmu_{2}}$ is given by $X_0\oplus (0)$ and $(0)\oplus X_0$. In particular, the $\Gamma$-representation given by $\Lambda(C,L)$ coincides with the standard rank two representation $V$ of $\Gm^{ 2}\rtimes C_2$, as defined in \cite{VisPGLp}*{pg. 19}.

Let $\lambda_i=c_i(\Lambda)$. Then by construction (see rule (a) in the proof of \cite{VisPGLp}*{Proposition 6.4}) we have that $\phi(c_2)=\lambda_2$. We claim that $\phi(c_1)=\lambda_1+\xi$. Once we verify this statement, we are done.

By \cite{VisPGLp}*{Lemma 6.5} this can be verified by checking that $(\lambda_1+\xi)$ is sent to zero by the ring homomorphism
\[ \CH^*(\cB\left(\Gm^{ 2}\rtimes \C_2\right))\longrightarrow \CH^*(\cB\C_2) \]
induced by the inclusion $\C_2\hookrightarrow \Gm^{ 2}\rtimes \C_2$. The standard representation $V$, regarded as a $\C_2$-representation, splits into the direct sum of the trivial representation and the sign representation. Therefore, $\lambda_1$ is sent to $\xi$, which implies that $\lambda_1+\xi\mapsto 0$.
\end{proof}
The following Corollary will be needed in Section \ref{sec:concrete} and it can be skipped at first reading.
\begin{cor}\label{cor:nu}
    Consider the pushforward morphism 
    \[\nu_*\colon\CH^*(\cB\Gm^{2})\simeq\ZZ[\alpha,\beta]\longrightarrow\CH^*(\cB\Gm^{2}\rtimes \C_2)\]
    induced by the finite morphism $\nu\colon\cB\Gm^{2}\to \cB(\Gm^{2}\rtimes \C_2)$. Then we have
    \begin{align*}
        \nu_*1&=2 \\
        \nu_*\alpha=\nu_*\beta&=\lambda_1+\xi.
    \end{align*}
\end{cor}
\begin{proof}
    The morphism $\nu$ is finite of degree $2$, hence $\nu_*1=2$. What we call $\nu_*$ here is the transfer morphism of \cite{VisPGLp}*{ pg.22}, and hence we have
    \[ \nu_*\alpha=\nu_*\beta=\phi(\alpha+\beta)=\phi(c_1). \]
    The morphism $\phi$ above was introduced in the proof of Proposition \ref{prop:CH Lcaluno}, where we also showed that $\phi(c_1)=\lambda_1+\xi$. This concludes the proof.
\end{proof}

\subsection{Computation of $\CH^*(\Pcal)$}
We are ready to prove the main result of the Section.
\begin{prop}\label{prop:CH Lcal}
   We have
   \[ \CH^*(\Pcal)\simeq \ZZ[\lambda_1,\lambda_2,\delta_1,\eta]/(2\eta,\eta(\lambda_1\delta_1+\eta)) \]
   where the cycles $\lambda_j$ are the Chern classes of $\Lambda$, the cycle $\delta$ is the class of $\Pcal_1$ and $\eta$ is the pushforward of $\xi$ along $i\colon\Pcal_1\hookrightarrow \Pcal$. Moreover, we have that $i_*\xi^2=\lambda_1\eta$.
 \end{prop}
 The proof of Proposition \ref{prop:CH Lcal} relies on the following result.
 \begin{lm}\label{lm:Atiyah}
     Let $X$ be a smooth variety endowed with the action of a group $G$, and let $Y\xhookrightarrow{i} X$ be a smooth, closed and $G$-invariant subvariety, with normal bundle $\Ncal$. Suppose that $c_{\rm top}^G(\Ncal)$ is not a zero-divisor in $\CH^*_G(Y)$. Then the following diagram of rings is cartesian:
     \[\xymatrix{
     \CH^*_G(X) \ar[r]^{i^*} \ar[d]^{j^*} & \CH^*_G(Y) \ar[d]^{q} \\
     \CH^*_G(X\smallsetminus Y) \ar[r]^{p} & \CH^*(Y)/(c_{\rm top}^G(\Ncal))
     }\]
     where the bottom horizontal arrow $p$ sends the class of a variety $V$ to the equivalence class of $i^*\xi$, where $\xi$ is any element in the set $(j^*)^{-1}([V])$.
 \end{lm}
 \begin{proof}
 It is enough to prove the result in the case of schemes, i.e. assuming $G=\{{\rm id}\}$. The equivariant case follows by applying the standard equivariant approximation technique, as in \cite{EG}.
 
 First we verify that $p$ is well defined: given two elements $\xi$ and $\xi'$ in the set $(j^*)^{-1}([V])$, then the localization exact sequence shows that $\xi-\xi'=i_*\eta$ for some $\eta$ in $\CH^*(Y)$. This implies that 
 \[p(i^*\xi - i^*\xi')=p( i^*i_*\eta )=p(c_{\rm top}(\Ncal)\eta)=0, \]
 so $p$ is well defined.
 
 Consider a subvariety $V\subset X$. Then $p(j^*[V])=q(i^*[v])$ by construction, therefore there is a map from $\CH^*(X)$ to the cartesian product of the two rings. Observe that the localization exact sequence
 \[ \CH_*(Y) \longrightarrow \CH_*(X) \longrightarrow \CH_*(X\smallsetminus Y) \longrightarrow 0 \]
 is exact on the left too: indeed, if $i_*\xi=0$, then $0=i^*i_*\xi=c_{\rm top}(\Ncal)\xi$, which implies $\xi=0$. This immediately implies that the map from $\CH^*(X)$ to the cartesian product of the two rings is both injective and surjective, thus finishing the proof.
 \end{proof}
 
 \begin{proof}[Proof of Proposition \ref{prop:CH Lcal}]
 First let us recall that $\Pcal_0\simeq\cB\GL_2$ (Proposition \ref{prop:Lcalzero iso BGLtwo}), hence we know the Chow ring of this open substack. We also computed the Chow ring of $\Pcal_1$ in Proposition \ref{prop:CH Lcaluno}.
 
 Consider the closed immersion of smooth algebraic stacks $i\colon\Pcal_1\hookrightarrow\Pcal$. Both of them are quotient stacks, and by Proposition \ref{prop:CH Lcaluno} the top Chern class of the normal bundle of $\Pcal_1$ in $\Pcal$ is not a zero divisor, hence we can apply Lemma \ref{lm:Atiyah}. We get the following cartesian diagram of rings
 \begin{equation}\label{eq:diagram atiyah}\xymatrix{
 \CH^*(\Pcal) \ar[r] \ar[d] & \ZZ[\lambda_1,\lambda_2,\xi,\delta]/(2\xi,\xi(\lambda_1+\xi)) \ar[d] \\
 \ZZ[\lambda_1,\lambda_2] \ar[r] & \ZZ[\lambda_1,\lambda_2,\xi]/(2\xi,\xi(\lambda_1+\xi))
 }\end{equation}
 and the localization short exact sequence
 \[ 0\longrightarrow \ZZ[\lambda_1,\lambda_2,\xi,\delta]/(2\xi,\xi(\lambda_1+\xi)) \longrightarrow \CH^*(\Pcal) \longrightarrow \ZZ[\lambda_1,\lambda_2] \longrightarrow 0, \]
where all the morphisms involved are morphisms of $\CH^*(\cB\GL_2)$-modules. 
From this we deduce that $\CH^*(\Pcal)$ is generated as a ring by $\lambda_1$, $\lambda_2$, $\delta_1:=i_*1$ and $i_*\xi^k$ for $k>0$. Actually $i_*\xi^k=i_*\xi\lambda_1^{k-1}$ for $k>1$, hence we only need $\eta:=i_*\xi$ to generate $\CH^*(\Pcal)$ as a ring.

To compute the relations, we have to check what polynomials in the generators above are sent to zero by both the top horizontal and the left vertical morphism in the diagram (\ref{eq:diagram atiyah}). The pullback morphism $i^*$ works as follows:
\begin{align*}
    &i^*\lambda_i=\lambda_i \\
    &i^*\delta_1=\delta_1 \\
    &i^*\eta=i^*i_*\xi=\delta_1\xi.
\end{align*}
Therefore, the image of $i^*$ in $\CH^*(\Pcal_1)$ is the subring generated by $\lambda_1$, $\lambda_2$, $\delta_1$ and $\delta_1\xi$, the ideal of relations can be deduced by looking at the intersection of the ideal of relations of $\CH^*(\Pcal_1)$ with the subring $\im(i_*))$. We deduce that the ideal of relations in the subring is generated by
\[ 2\delta_1\xi, \quad \delta_1\xi(\lambda_1\delta_1 + \delta_1\xi). \]
The polynomials whose pullback coincides with one of the generators above are respectively
\[ 2\eta, \quad \eta(\lambda_1\delta_1+\eta) \]
and both of them are in the kernel of the pullback $\CH^*(\Pcal)\to\CH^*(\Pcal_0)$. This implies that these two elements generate the ideal of relations in $\CH^*(\Pcal)$. Moreover, from the projection formula we have 
\[ i_*\xi^2=i_* \lambda_1\xi=\lambda_1\cdot i_*\xi =\lambda_1\eta. \]
This concludes the proof.
 \end{proof}
 
\section{Relations in the Chow ring of $\overline{\Mcal}_2$: abstract characterization}\label{sec:abstract}
The purpose of this Section is to give a first description of the integral Chow ring of $\overline{\Mcal}_2$ in terms of generators and relations, which looks as follows (Corollary \ref{cor:Chow M2bar abs}) :
\[ \CH^*(\overline{\Mcal}_2) \simeq \ZZ[\lambda_1,\lambda_2,\delta_1]/(2\delta_1(\lambda_1+\delta_1),\delta_1^2(\lambda_1+\delta_1),[\Dcal^{3}],\zeta, [\Dcal^{3,3}]),  \]
where the $\lambda_i$ are the Mumford lambda classes and $\delta_1$ is the boundary divisor $[\Delta_1]$.
Observe that some relations are not explicitly computed and instead they are formulated in terms of fundamental classes of closed substacks in $\Vcal$. More precisely, the stacks $\Dcal^3$ and $\Dcal^{3,3}$ appearing in the formula above are the ones introduced respectively in Definition \ref{def:stack D3} and \ref{def:stack D33}, whereas $\zeta$ is the pushforward of a certain first Chern class along $\Dcal_3\hookrightarrow\Vcal$.

In Section \ref{sec:concrete} we will complete the computation of the integral Chow ring of $\overline{\Mcal}_2$ by expressing all the relations in terms of the classes $\lambda_1$, $\lambda_2$ and $\delta_1$.

\subsection{Relations coming from $\Dcal^3$}\label{subsec:relations from W}
Recall that the stack $\Ccal$ is the universal curve over $\Rcal$ (Definition \ref{def:stack C}). Its objects are pairs $(C\to S,p)$, where $C\to S$ is a family of twisted conics, and $p:S\to C$ is a section of $C\to S$. Observe that the morphism of stacks $\Ccal\to \Rcal$ is non-representable. The smooth locus of $\Ccal\to\Rcal$ is denoted $\Ccal^{\rm sm}$ and in particular we have that $\Ccal^{\rm sm}\to \Rcal$ is representable, but not proper.

The restriction of $\Ccal$ (resp. $\Ccal^{\rm sm}$) to $\Rcal_i$ is denoted $\Ccal_i$ (resp. $\Ccal^{\rm sm}_i$). Moreover, the pullback of $\Ccal$ (resp. $\Ccal^{\rm sm}$) to $\Vcal$ is denoted $\Ccal_{\Vcal}$ (resp. $\Ccal^{\rm sm}_\Vcal$). We adopt a similar notation of the pullback of $\Ccal$ to $\Vcal_i$, $\Vcal\smallsetminus \Zcal$, etc., where $\Zcal$ is the closed substack introduced in Definition \ref{def:stack Z}.

Recall also that the universal polarized twisted conic $\pi\colon\Ccal_\Pcal\to\Pcal$ has a universal polarization $\Lcal$ and that in \ref{subsec:chow P1} we set $\Lambda:=\pi_*\Lcal$. In particular, there is a surjective morphisms of vector bundles $\pi^*\Lambda\to\Lcal$ over $\Ccal_{\Pcal}$.

We denote by $\Wcal$ the closed substack of $\Ccal_{\Vcal\smallsetminus \Zcal}^{\rm sm}$ whose objects are quadruples $(C\to S, L, s, p)$ such that $s$ vanishes on the image of the section $p:S\to C$ with order $\geq 3$. This stack has been already introduced in \ref{subsec:closed substacks} when we were defining $\Dcal^3$.

The morphism $f\colon\Wcal\to \Vcal\smallsetminus\Zcal$ factorizes through $\Dcal^3$, and it is actually an isomorphism over the complement of $\Dcal^{3,3}$.
\begin{prop}\label{pr:Wcal}
   We have:
   \begin{enumerate}
       \item The morphism $f\colon\Wcal\to\Vcal\smallsetminus\Zcal$ is representable and proper.
       \item Let $\alpha$ be the first Chern class of the pullback of the line bundle $\Lcal$ along $\Wcal\to\Ccal_{\Pcal}$. Then the image of the pushforward morphism \[f_*\colon\CH^*(\Wcal)\longrightarrow\CH^*(\Vcal\smallsetminus\Zcal)\]
       is generated as an ideal by $f_*1$ and $f_*\alpha$.
   \end{enumerate}
\end{prop}
\begin{proof}
The inclusion $\Wcal\hookrightarrow\Ccal_{\Vcal\smallsetminus\Zcal}^{\rm sm}$ is representable. The same thing holds for the morphism $\Ccal_{\Vcal\smallsetminus\Zcal}^{\rm sm}\to\Vcal\smallsetminus\Zcal$, because it is the pullback of $\Ccal^{\rm sm}\to\Rcal$ along $\Vcal\smallsetminus\Zcal\to \Rcal$. Therefore, the composition $f\colon\Wcal\to\Vcal\smallsetminus\Zcal$ is representable.

The morphism $\Ccal_{\Vcal}\to\Vcal$ is proper, although not representable. Let $\Scal\subset\Ccal_{\Vcal}$ be the vanishing locus of the universal section of the line bundle $\Lcal_{\Vcal}^{\otimes 2}\otimes \omega_{\Ccal_{\Vcal}/\Vcal}^{\otimes (-2)}$. Then $\Scal\to\Vcal$ is proper, and so it is $\Scal_{\Vcal\smallsetminus\Zcal}\to\Vcal\smallsetminus\Zcal$. Observe that $\Wcal$ is a closed substack of $\Scal_{\Vcal\smallsetminus\Zcal}$: this implies that $f\colon\Wcal\to\Vcal\smallsetminus\Zcal$ is proper too. We have proved (1).

Consider the morphism $\Ccal^{\rm sm}_{\Vcal}\to\Ccal^{\rm sm}_{\Pcal}$: this is a vector bundle, because it is the pullback of $\Vcal\to\Pcal$ to $\Ccal_{\Pcal}$. Moreover $\Wcal$ in $\Ccal^{\rm sm}_{\Vcal}$ is a vector subbundle, as we have seen in \ref{subsec:closed substacks}. By composing together all the induced pullbacks at the level of Chow rings, we obtain the following surjective homomorphism:
\begin{equation}\label{eq:surjective} \CH^*(\Ccal^{\rm sm}_{\Pcal})\longrightarrow \CH^*(\Wcal). \end{equation}

Consider the Cartesian diagram
\[\xymatrix{
\Ccal^{\rm sm}_{\Pcal}\times_{\Pcal}\Ccal^{\rm sm}_{\Pcal} \ar[d]^{\pr_1} \ar[r]^{\quad\pr_2} & \Ccal^{\rm sm}_{\Pcal} \ar[d]^{\pi} \\
\Ccal^{\rm sm}_{\Pcal} \ar[r]^{\pi} & \Pcal
}\]
Let $\Delta$ be the diagonal in $\Ccal^{\rm sm}_{\Pcal}\times_{\Pcal}\Ccal^{\rm sm}_{\Pcal}$. Then we have an exact sequence of vector bundles
\[ 0\to \pr_{1*}(\pr_2^*\Lcal\otimes\Ocal(-\Delta)) \to \pr_{1*}\pr_2^*\Lcal\simeq\pi^*\Lambda \to \pr_1*(\pr_2^*\Lcal|_{\Delta}) \to 0. \]
Let $\alpha$ be the first Chern class of the line bundle on the left, and call $\beta$ be the first Chern class of the line bundle on the right, so that we have $\alpha+\beta=\lambda_1$ and $\alpha\beta=\lambda_2$.
We claim that 
\[\CH^*(\Ccal^{\rm sm}_{\Pcal})\simeq \ZZ[\alpha,\beta,\delta_1] \]
This, together with the surjectivity of (\ref{eq:surjective}), would imply that $\CH^*(\Wcal)$ is generated by the restriction of $\alpha$, $\beta$ and $\delta_1$ and hence the image of
\[ f_*\colon\CH^*(\Wcal) \longrightarrow \CH^*(\Vcal\smallsetminus\Zcal) \] 
is generated by $f_*1$ and $f_*\alpha$: indeed, this would be a consequence of the formulas $\delta_1=f^*\delta_1$, $\beta=f^*\lambda_1-\alpha$ and
\[
\begin{split}
    f_*\alpha^n & =f_*(\alpha^n+\alpha^{n-1}\beta-\alpha^{n-1}\beta)\\
                & =f_*(\alpha^{n-1}\cdot f^*\lambda_1-\alpha^{n-2}\cdot f^*\lambda_2)\\
                & = \lambda_1\cdot f_*\alpha^{n-1} -\lambda_2\cdot f_*\alpha^{n-2}
\end{split}
\]
Let $\B_2\subset\GL_2$ be the Borel subgroup of upper-triangular matrices. Then the stack $\Ccal^{\rm sm}_{\Pcal_0}\simeq\Ccal_{\Pcal_0}$ is isomorphic to the classifying stack $\cB\B_2$: indeed, every object of $\Ccal_{\Pcal_0}$ is \'{e}tale locally isomorphic to the data of the family of curves $\PP^1_S\to S$ together with the line bundle $\Ocal(1)$ and the section $\{\infty\}\times S$, whose automorphism group is precisely $\B_2$.

We deduce then
\[\CH^*(\Ccal_{\Pcal_0})\simeq \CH^*(\cB\B_2)\simeq \ZZ[\alpha,\beta], \]
where $\alpha$ and $\beta$ are the Chern roots of the rank two vector bundle determined by the morphism $\cB\B_2\to \cB\GL_2$. 

On the other hand we have:
\[\Ccal^{\rm sm}_{\Pcal_1} \simeq \cB\left( \Gm\times (\Ga\rtimes\Gm) \times\Gm \right) \]
This follows from the fact that every object of $\Ccal^{\rm sm}_{\Pcal_1}$ is \'{e}tale locally isomorphic to the data of a pairs $\PP^1\cup_{0}\PP^1$ twisted at the node, together with the line bundle $\Ocal(\frac{1}{2}0)\cup_{0}\Ocal(\frac{1}{2}0)$ and the section $\{\infty\}\times S$ on the first $\PP^1$. Therefore:
\[ \CH^*(\Ccal^{\rm sm}_{\Pcal_1})\simeq \ZZ[\alpha,\beta,\delta_1] \]
with $\alpha$ and $\beta$ Chern roots of the rank two vector bundle induced by 
\[ \cB\left( \Gm\times (\Ga\rtimes\Gm) \times\Gm \right) \longrightarrow\cB\GL_2 \]
Now we apply Lemma \ref{lm:Atiyah} to obtain the desired description of $\CH^*(\Ccal^{\rm sm}_{\Pcal})$. This concludes the proof of (2).
\end{proof}
\subsection{Relations coming from $\Dcal^{3,3}$}
\begin{prop}\label{prop:Ucal}
    Let $I$ be the ideal in $\CH^*(\Vcal\smallsetminus\Zcal)$ generated by the cycles coming from $\Wcal$ (see Proposition \ref{pr:Wcal}).
    Then the image of the morphism
   \[ \CH^*(\Dcal^{3,3})\longrightarrow\CH^*(\Vcal\smallsetminus\Zcal)/I \]
   is generated as an ideal by $[\Dcal^{3,3}]$.
\end{prop}
\begin{proof}
Let $J$ be the image of $\CH^*(\Dcal^{3,3})\to\CH^*(\Vcal\smallsetminus\Zcal)/I$. We will show first that $2J=0$, and then that $9J\subset ([\Dcal^{3,3}])$. Putting these two claims together, we get the desired conclusion.

The first claim follows directly from the fact that the restriction of $\Wcal\to\Vcal\smallsetminus\Zcal$ over $\Dcal^{3,3}$ is finite of degree $2$. For the second one, consider the universal polarized twisted conic $\pi\colon\Ccal_{\Pcal}\to\Pcal$, with $\Lcal$ its universal polarization. Observe that the two vector bundles 
\[\pi_*(\Lcal^{\otimes 6}),\quad \Vcal=\pi_*\left(\Lcal^{\otimes 2}\otimes\omega_{\Ccal/\Pcal}^{\otimes (-2)}\right)\]
are isomorphic up to a twist by a line bundle of $\Pcal$. Henceforth, there is an isomorphism $\PP(\pi_*(\Lcal^{\otimes 6}))\simeq\PP(\Vcal)$ over $\Pcal$.

Let $\overline{\Zcal}$ and $\overline{\Dcal^{3,3}}$ be the projectivizations of $\Zcal$ and $\Dcal^{3,3}$. Then by \cite{VisPGLp}*{Lemma 4.1} the horizontal arrows in the following commutative diagram are surjective:
\[\xymatrix{ 
\CH^*(\overline{\Dcal^{3,3}}) \ar[r] \ar[d] & \CH^*(\Dcal^{3,3}) \ar[d] \\
\CH^*(\PP(\pi_*(\Lcal^{\otimes 6}))\smallsetminus\overline{\Zcal})\ar[r] & \CH^*(\Vcal\smallsetminus\Zcal) }\]
If we show that nine times the image of $\CH^*(\overline{\Dcal^{3,3}})\to\CH^*(\PP(\pi_*(\Lcal^{\otimes 6})))$ is contained in $([\overline{\Dcal^{3,3}}])$, we are done. 

Consider the closed embedding $g\colon \PP(\pi_*(\Lcal^{\otimes 2}))\to\PP(\pi_*(\Lcal^{\otimes 6}))$ that sends a section $s$ to $s^{\otimes 3}$ and observe that this induces an isomorphism of the domain with $\overline{\Dcal^{3,3}}$. Let $H$ and $K$ denote the hyperplane sections respectively of $\PP(\pi_*(\Lcal^{\otimes 6}))$ and $\PP(\pi_*(\Lcal^{\otimes 2}))$, so that in particular the cycles $K^j$ for $j=0,1,2$ generate the Chow ring of $\PP(\pi_*(\Lcal^{\otimes 2}))$ as $\CH^*(\Pcal)$-module. Then $\im(g_*)$ is generated as an ideal by elements of the form $g_*K^j$.

Observe that $g^*H=3K$, hence $g_*3^jK^j=H^j\cdot g_*1$. Therefore $9\cdot\im(g_*)$ is contained in $(g_*1=[\overline{\Dcal^{3,3}}])$. This concludes the proof.
\end{proof}

Putting all together, we deduce the following abstract characterization of the Chow ring of $\overline{\Mcal}_2$.
\begin{prop}\label{prop:CH M2bar abs}
   Using the same notation of Proposition \ref{prop:CH Lcal}, we have:
   \[ \CH^*(\overline{\Mcal}_2) \simeq \CH^*(\Vcal\smallsetminus\Zcal)/([\Dcal^3],\zeta,[\Dcal^{3,3}]) \]
   where $\zeta$ is the pushforward along $i\colon\Dcal^3 \hookrightarrow \Vcal\smallsetminus\Zcal$ of a Chern root of $i^*\Lambda$.
\end{prop}

\subsection{Computation of the class of $\Zcal$}
In this last Subsection we compute $\CH^*(\Vcal\smallsetminus\Zcal)$, so to obtain a description of $\CH^*(\overline{\Mcal}_2)$ that is somehow in between of the abstract characterization of Proposition \ref{prop:CH M2bar abs}
and the explicit final result.

\begin{prop}\label{prop:Vcal0 quotient}
   Let $V_0$ be the affine space of binary forms of degree $6$, regarded as a $\GL_2$-representation via the action
   \[ A\cdot f(x_0,x_1) := \det(A)^2 f(A^{-1}(x_0,x_1)). \]
   Then $\Vcal_0\simeq \left[ V_0/\GL_2  \right]$.
\end{prop}
\begin{proof}
This is an easy extension of \cite{VisM2}*{Proposition 3.1} .
\end{proof}

Let $V_1$ be the affine space of pairs of binary forms $(f,g)$ of degree $3$ satisfying $f(0,1)=g(0,1)$. If $(u',v',u,v,t,\sigma)$ denotes an element of the group $\Ga^2\rtimes\overline{\Gamma}$, we have an action of this group on $V_1$ defined by:
\[ (u',v',u,v,t) \cdot  (f,g) := (t^2f(u^{-2}x_0,x_1-u'u^{-2}x_0),t^2g(v^{-2}y_0,y_1-v'v^{-2}y_0)) \]
and $\sigma$ exchanges $f$ and $g$. Observe that the formula above induces a well defined action of $\Ga^2\rtimes\Gamma$, hence we can regard $V_1$ as a $\Ga^2\rtimes\Gamma$-representation.

\begin{prop}\label{prop:Vcal1 quotient}
   We have $\Vcal_1\simeq [V_1/\Ga^2\rtimes\Gamma]$. Moreover, the pullback along the morphism $[V_1/\Gamma]\to \Vcal_1$ induces an isomorphism at the level of Chow rings.
\end{prop}
\begin{proof}
Let $C$ be the usual curve $\PP^1\cup_{0}\PP^1$ twisted in the node, with $L:=\Ocal(\frac{1}{2}0)\cup_{0}\Ocal(\frac{1}{2}0)$.

We know from Proposition \ref{prop:Lcaluno iso BGamma} that the morphism $f\colon\Spec(k)\to\Pcal_1$ induced by $(C,L)$ is a $\Ga^2\rtimes\Gamma$-torsor, hence $f^*\Vcal_1$ is a $\Ga^2\rtimes\Gamma$-representation and $\Vcal_1\simeq [f^*\Vcal_1/\Ga^2\rtimes\Gamma]$. By definition $f^*\Vcal_1=\H^0(C,L^{\otimes 2}\otimes\omega_{C}^{\otimes (-2)})$.

If $D$ denotes the curve $\PP^1\cup_{0}\PP^1$, we have seen in the proof of Proposition \ref{prop:Lcaluno iso BGamma} that there is a morphism $g:D\to C$ obtained by first taking the quotient of $D$ with respect to the action of $\mmu_{2}$ and then rigidifying the points at infinity. Moreover, the line bundle $M:=\Ocal(0)\cup_{0}\Ocal(0)$ is equal to the pullback of $L$ along this morphism. Therefore we have:
\[g^*f^*\Vcal_0 \simeq \H^0(D, M^{\otimes 2}\otimes\omega_{D}^{\otimes (-2)})^{\mmu_{2}} \]
If $x_1^2$ denotes the generator of $M^{\otimes 2}$ near the node, then an element
\[ (u',v',u,v,t,\sigma) \in \Ga^2\rtimes\overline{\Gamma}\]
acts on $x_1^{2}$ by sending it to $t^2 x_1^{2}$, and $\sigma$ exchange $x_i$ and $y_i$.

The dualizing sheaf is generated around the node by the log differential that is equal to $\frac{dx}{x}$ on one component and $\frac{dy}{y}$ on the other. This generator is invariant with respect to the $\Ga^2\rtimes\overline{\Gamma}$-action, hence the same thing must be true for the local generator of $\omega_{D}^{\otimes (-2)}$ near the node.

We deduce that $\Ga\rtimes\overline{\Gamma}$ acts on the generator $e$ of $\left(M^{\otimes 2}\otimes\omega_{D}^{\otimes -2}\right)_{0}$ by multiplication by $t^2$. An element of $\left(M^{\otimes 2}\otimes\omega_{D}^{\otimes -2}\right)_{0}$ is of the form $(f(x),g(y))\cdot e$, where $f$ and $g$ are polynomials satisfying $f(0)=g(0)$. In order for these elements to extend to the whole of $D$, the degrees of $f$ and $g$ must be $\leq 6$. In other terms:
\[ \H^0(D, M^{\otimes 2}\otimes\omega_{D}^{\otimes (-2)})=\langle (x_1^6,y_1^6)\cdot e, (x_0x_1^5,0)\cdot e, (0,y_0y_1^5)\cdot e,\ldots,(x_0^6,0)\cdot e,(0,y_0^6)\cdot e \rangle \]
The subspace of $\mmu_{2}$-invariant elements is then generated by $e$, $(x_0^{2i}x_1^{6-2i},0)\cdot e$ and $(0,y_0^{2i}y_1^{6-2i})$, where $i$ goes from $1$ to $3$.

By construction $\Ga\rtimes\overline{\Gamma}$ acts on $x_0$ and $x_1$, regarded as rational functions, as:
\[ (u',v',u,v,t)\cdot x_0=u^{-1}x_0,\quad (u',u,v',v,t)\cdot x_1= x_1-u'u^{-2}x_0\]
and $\sigma$ exchanges $x_i$ and $y_i$. A similar description holds for $y_i$, with $v$ in place of $u$.

Putting all together, we get the claimed description of $f^*\Vcal_1$. The fact that $[V_1/\Gamma]\to\Vcal_1$ induces an isomorphism at the level of Chow rings can be deduced from \cite{MolVis}*{Lemma 2.3}.
\end{proof}

\begin{prop}\label{prop:class Zcal}
   We have:
   \begin{enumerate}
       \item $[\Zcal]=\lambda_1+\delta_1+\xi$ in $\CH^1(\Vcal_1)$.
       \item $[\Zcal]=\eta+\delta_1(\lambda_1+\delta_1)$ in $\CH^2(\Vcal)$.
   \end{enumerate}
\end{prop}
\begin{proof}
The stack $\Zcal$ is a vector subbundle of $\Vcal_1$ of codimension $1$. We know from Proposition \ref{prop:Vcal1 quotient} that $\Vcal_1\simeq [V_1/\Ga^2\rtimes\Gamma]$, where $V_1$ is the $\Ga^2\rtimes\Gamma$-representation consisting of pairs $(f,g)$ where $f$ and $g$ are both binary forms of degree three, satisfying $f(0,1)=g(0,1)$.

The stack $\Zcal$ is then isomorphic to $[U_1/\Gamma]$, where $U_1$ is the $\Ga^2\rtimes\Gamma$-subrepresentation generated by pairs $(f,g)$ with $f(0,1)=g(0,1)=0$.

From now on, we will work with the group $\Gamma$ rather than $\Ga^2\rtimes\Gamma$. The group $\Gamma$ is simpler to handle and the final result will not be affected by this change, as Proposition \ref{prop:Vcal1 quotient} shows.

Let $e^{\vee}$ be the functional on $V_1$ dual to $e$: then for every $\gamma$ in $\Gamma$ we have $\gamma\cdot e^{\vee}=\chi_e(\gamma)^{-1}e^{\vee}$, where $\chi$ is the character
\[ \chi_e(u^{-1}t,v^{-1}t,uv,\sigma) = t^2 = (u^{-1}t)(v^{-1}t)(uv). \]

Recall from the proof of Proposition \ref{prop:CH Lcaluno} that the first Chern class of the rank $1$ representation determined by the character $uv$ is equal to $\delta_1$.

Let $V$ be the rank $2$ representation induced by the injective homomorphism $\Gamma\hookrightarrow \GL_2$, where $\Gm^{ 2}\times\{\id\}\times\{1\}$ is sent to subtorus of diagonal matrices, the generator $\sigma$ of $\C_2$ is sent to the transposition matrix and the last copy of $\Gm$ is sent to $1$.

We have seen in Proposition \ref{prop:Vcal1 quotient} that $\lambda_1$ is the first Chern class of $\det(V)$: the associated character is $u^{-1}v^{-1}t^2\otimes\sigma$, where $\sigma$ is the non-trivial character of $\C_2$. Therefore $u^{-1}v^{-1}t^2$ is the character associated to $\det(V)\otimes\sigma$, whose first Chern class is $\lambda_1+\xi$.

It follows from \cite{EF}*{Lemma 2.4} that the equivariant class of $U_1$ in $\CH^*_{\Gamma}(V_1)$ is equal to the first Chern class of the representation determined by $\chi_e$, which for what we have just observed coincides with $\lambda_1+\xi+\delta_1$. This proves (1).

Let $i\colon\Vcal_1\to\Vcal$ be the closed embedding. Then it follows from Proposition \ref{prop:CH Lcal} that
\[ i_*(\lambda_1+\delta_1+\xi)=i_*i^*(\lambda_1+\delta_1)+i_*\xi=\delta_1(\lambda_1+\delta_1)+\eta. \]
This proves (2).
\end{proof}

\begin{cor}\label{cor:Chow M2bar abs}
We have
\[ \CH^*(\overline{\Mcal}_2) \simeq \ZZ[\lambda_1,\lambda_2,\delta_1]/(2\delta_1(\lambda_1+\delta_1),\delta_1^2(\lambda_1+\delta_1),[\Dcal^{3}],\zeta, [\Dcal^{3,3}]),  \]
where the $\lambda_i$ are the Mumford lambda classes and $\delta_1$ is the boundary divisor $[\Delta_1]$.
\end{cor}

\section{Final computations}\label{sec:concrete}
In this Section we assume for most of the time that the ground field has characteristic zero. We drop this assumption only in the very last Subsection. The main result is the computation of the integral Chow ring of $\overline{\Mcal}_2$: we first complete this computation over fields of characteristic zero (Theorem \ref{thm:Chow M2bar}) and then we extend this result over fields of characteristic $\neq 2,3$ (Theorem \ref{thm:Chow M2bar ii}).

\subsection{Equivariant intersection theory of $T$-schemes}
We start with two technical results.
\begin{lm}[\cite{EF}*{Lemma 2.4}]\label{lm:class hypersurface}
    Given a $T$-representation $V$, let $Y\subset\PP(V)$ be a $T$-invariant hypersurface of degree $d$, whose homogeneous equation is $f(x)=0$. Let $\chi$ be the character of $T$ such that $f(\tau^{-1}\cdot x)=\chi(\tau)f(x)$ for every $\tau$ in $T$. Then
    \[ [Y]_T = dh + c_1^T(E_\chi), \]
    where $h=c_1^T(\Ocal(1))$ and $E_\chi$ is the rank one $T$-representation determined by $\chi$.
\end{lm}
The following Theorem is known as the explicit localization formula.
\begin{prop}[\cite{EGLoc}*{Theorem 2}]\label{prop:loc}
    Let $f\colon X\to Y$ be a proper $T$-equivariant morphism between smooth $T$-schemes, and suppose that $\CH_T^*(Y)$ is free as $\CH_T^*$-module. Then
    \[ f_*\xi=\sum_{F\subset X^T} \frac{f_*(i_F^*\xi)}{c_{top}^T(N_F)}, \]
    where the sum is taken over all the irreducible components $i_F:F\hookrightarrow X$ of the fixed locus $X^T$, and $N_F$ denotes the normal vector bundle of $F$ in $X$.
\end{prop}

Let $T\simeq\Gm^2$ be a split torus. If $\gamma=(\gamma_1,\gamma_2)$ is a pair of characters of $T$, we denote by $E_{\gamma}$ the rank two representation on which $T$ acts with weights $\gamma_1$ and $\gamma_2$.

Consider the following three equivariant morphisms:
\begin{align*}
    i\colon\PP(E_{\gamma}^{\vee})\longrightarrow \PP(\Sym^3E_{\gamma}^{\vee}),\quad &f\longmapsto f^3,\\
    \rho_1\colon\PP(E_{\gamma}^{\vee})\times\PP(\Sym^3E_{\gamma}^{\vee})\longrightarrow\PP(\Sym^6E_{\gamma}^{\vee}),\quad & (f,g)\longmapsto f^3g.
\end{align*}
We want to compute the pushforward along the maps above of some cycles in the $T$-equivariant Chow rings.
\begin{prop}\label{prop:formulas}
    We have:
    \begin{align*}
    i_*1 &= 3(h-2\gamma_1-\gamma_2)(h-\gamma_1-2\gamma_2)\\
    \rho_{1*} 1 &= 3(4h^2-24h(\gamma_1+\gamma_2)+20(2\gamma_1+\gamma_2)(\gamma_1+2\gamma_2)-36\gamma_1\gamma_2) \\
    \rho_{1*} h_1 &= h^3-3(\gamma_1+\gamma_2)h^2 + h(2(\gamma_1+\gamma_2)^2-44\gamma_1\gamma_2)+108\gamma_1\gamma_2(\gamma_1+\gamma_2) \\
    \rho_{2*}1 &= 9(h-5\gamma_1-\gamma_2)(h-4\gamma_1-2\gamma_2)(h-2\gamma_1-4\gamma_2)(h-\gamma_1-5\gamma_2)
    \end{align*}
\end{prop}
\begin{proof}
    All the formulas above can be obtained by applying the localization formula of Proposition \ref{prop:loc}. For this we need the following ingredients:
    the fundamental class of the $T$-fixed points of $\PP(\Sym^d E_{\gamma}^{\vee})$, the $T$-equivariant top Chern class of the tangent space at the fixed points, and the class of the hyperplane class restricted to the equivariant Chow ring of the fixed points.
    
    A $T$-fixed point in $\PP(\Sym^d E_{\gamma}^{\vee})$ corresponds to a homogeneous monomial $x_0^{i}x_1^{d-i}$, where $0\leq i \leq d$. This point is a complete intersection of $d$ invariant hypersurfaces of equation $(x_0^{j}x_1^{d-j})^{\vee}=0$, for $j\neq i$.
    Applying the formula of Lemma \ref{lm:class hypersurface} we obtain
    \[ [(x_0^ix_1^{d-i})]_T = \prod_{0\leq j\leq d, j\neq i} (h-j\gamma_1-(d-j)\gamma_2). \]
    A similar expression holds for fixed points in $\PP(\Sym^d E_{\gamma}^{\vee})\times\PP(\Sym^e E_{\gamma}^{\vee})$, namely
    \[ [(x_0^ix_1^{d-i},x_0^{k}x_1^{e-k})]_T = \prod_{0\leq j\leq d, j\neq i} (h-j\gamma_1-(d-j)\gamma_2) \cdot  \prod_{0\leq \ell\leq e, \ell\neq k} (h-\ell\gamma_1-(e-\ell)\gamma_2). \]
    Observe that 
    \[ T_{(x_0^ix_1^{d-i})} = \left\langle  \frac{(x_0^jx_1^{d-j})}{(x_0^ix_1^{d-i})}\text{ for } j\neq i  \right\rangle \]
    and therefore
    \begin{align*}
        &c_{top}^T(T_{(x_0^ix_1^{d-i})}) = \prod_{0\leq j\leq d, j\neq i} (j-i)(\gamma_2-\gamma_1) = (-1)^ii!(d-i)!(\gamma_2-\gamma_1)^d \\
        &c_{top}^T(T_{(x_0^ix_1^{d-i},x_0^{k}x_1^{e-k})}) = (-1)^{i+k}i!(d-i)!k!(e-k)!(\gamma_1-\gamma_2)^{d+e}
    \end{align*} 
    Finally, a generator for the fiber of $\Ocal(1)$ over $(x_0^ix_1^{d-i})$ is given by $(x_0^ix_1^{d-i})^{\vee}$, hence
    \[ h|_{\CH^*_T((x_0^ix_1^{d-i}))}= i\gamma_1 + (d-i)\gamma_2 \]
    We can substitute the elements above in the expression given by localization formula. After some straightforward computations, we obtain the desired results.
\end{proof}

\subsection{Computation of $[\Dcal^3]$ and $\zeta$}
The main goal of this Subsection is to compute the relations in $\CH^*(\overline{\Mcal}_2)$ corresponding to the cycle class of $\Dcal^3$ and the element $\zeta$ which appear in Corollary \ref{cor:Chow M2bar abs}.
\begin{prop}\label{prop:class D3}
    We have $[\Dcal^3]=24\lambda_1^2-48\lambda_2$ in $\CH^2(\Vcal\smallsetminus\Zcal)$.
\end{prop}

\begin{proof}
    We first prove that $[\Dcal^3_0]=24\lambda_1^2-48\lambda_2$ in $\CH^2(\Vcal_0)$. We adopt here the notation of Proposition \ref{prop:Vcal0 quotient}.
    
    Observe that $\PP(\Vcal_0)\simeq [\PP(V_0)/\GL_2]$. Moreover, if we denote $\overline{\Dcal_0^3}$ the projectivization of $\Dcal^{3}_0$ and $h$ the hyperplane class of $\PP(\Vcal_0)$, then by \cite{VisPGLp}*{Lemma 4.1} the cycle $[\Dcal^3_0]$ is equal to the cycle $[\overline{\Dcal^3_0}]$ after substituting $h$ with $2\lambda_1$.
    
    By \cite{EG}*{Proposition 6} we can identify $\CH_{\GL_2}^*(\PP(V_0))$ with the subring of symmetric elements of $\CH_T^*(\PP(V_0))$, where $T=\Gm^{2}$ and $V_0=\Sym^6 E_{t_1+t_2}^{\vee}$, a symmetric power of the dual of the representation on which $T$ acts diagonally. In particular, $\lambda_1=t_1+t_2$ and $\lambda_2=t_1t_2$.
    
    Therefore, we can apply Proposition \ref{prop:formulas} to compute $[\overline{\Dcal^3_0}]=\rho_{1*}1$: after straightforward computations we get the desired result.
    
    From \cite{Mum}*{\S 5} we know that the relation ${\rm ch}_2(\Lambda) = 24\lambda_1-48\lambda_2=0$ holds in $\CH^*(\overline{\Mcal}_2)$ up to torsion. As already observed in \cite{Lars}*{Remark 6.3}, it follows from \cite{Pap} that in characteristic zero this relation holds with integral coefficients. The characterization of $\CH^*(\overline{\Mcal}_2)$ given in Corollary \ref{cor:Chow M2bar abs} implies that Mumford's relation must be a multiple of $[\Dcal^3]$ in $\CH^*(\Vcal\smallsetminus\Zcal)$.
    
    On the other hand, our earlier computation of $[\Dcal^3_0]$ is equal to the restriction of Mumford's relation to $\CH^*(\Vcal_0)$, hence $[\Dcal^3]$ must be actually equal to $24\lambda_1^2-48\lambda_2$.
\end{proof}

\begin{prop}\label{prop:class zeta}
    Let $\zeta$ be the pushforward along $i\colon\Dcal^3\hookrightarrow\Vcal\smallsetminus\Zcal$ of a Chern root of $i^*\Lambda$. Then 
    $ \zeta=20\lambda_1\lambda_2-4\delta_1\lambda_2$ in $\CH^3(\Vcal\smallsetminus\Zcal)/([\Dcal^3])$.
\end{prop}

\begin{proof}
    We adopt here the same strategy used to compute $[\Dcal^3]$ in Proposition \ref{prop:class D3}, i.e. we compute the restriction of $\zeta$ to $\CH^*(\Vcal_0)/([\Dcal^3])$ and then we use the Mumford relations to deduce our formula.
    
    Consider again the $T$-equivariant proper morphism of projective spaces
    \[ \rho_{1*}\colon\PP(E_{\gamma}^{\vee})\times\PP(\Sym^3E_{\gamma}^{\vee})\longrightarrow\PP(\Sym^6E_{\gamma}^{\vee}),\quad  (f,g)\longmapsto f^3g, \]
    with $\gamma=(t_1,t_2)$. The subgroup of $\im(\rho_{1*})$ of degree $3$ is generated, up to multiples of $\rho_{1*}1$, by $\rho_{1*}h_1$, and by Proposition \ref{prop:formulas} we know that
    \[ \rho_{1*}h_1 = h^3-3\lambda_1 h^2 + (2\lambda_1^2-44\lambda_2)h + 108    \lambda_1\lambda_2. \]
    
    Let $i_0\colon\Dcal^3_0\hookrightarrow \Vcal_0$ be the restriction of $i$ to $\Vcal_0$. Then Proposition \ref{pr:Wcal} implies that the degree $3$ subgroup of $\im(i_{0*})$ is generated, up to  multiples of $[\Dcal^3]$, by a single element $\zeta_0$.
    
    To compute the pullback of $\rho_{1*}h_1$ to $\CH^3(\Vcal_0)$ one has to substitute $h$ with $2\lambda_1$, which gives back $20\lambda_1\lambda_2$. This element gives by construction a generator of $\im(i_*)$, which is not a multiple of $[\Dcal_3]$. Therefore, this cycle must coincide with $\zeta_0$ in $\CH^*(\Vcal_0)/([\Dcal^3])$.
    
    Mumford proved in \cite{Mum}*{\S 5} that the relation $20\lambda_1\lambda_2-4\delta_1\lambda_2=0$ holds up to torsion in $\CH^*(\overline{\Mcal}_2)$. By the same argument of Proposition \ref{prop:class D3}, i.e. the main result of \cite{Pap}, this relation must hold integrally. Moreover, this relation is independent of the other relations in $\CH^*(\overline{\Mcal_2})$ that we have found so far, which implies that it must be a multiple of $\zeta$ in $\CH^*(\Vcal\smallsetminus\Zcal)/([\Dcal^3])$, by Proposition \ref{prop:CH M2bar abs}.
    
    On the other hand, our earlier computation shows that, once both relations are restricted to $\CH^*(\Vcal_0)/([\Dcal_3^0])$, they coincide: this forces the equality $\zeta=20\lambda_1\lambda_2+4\delta_1\lambda_1$, up to multiples of $[\Dcal^3]$.
\end{proof}
\subsection{Computation of $[\Dcal^{3,3}]$}
\begin{prop}\label{prop:class D33}
    We have $[\Dcal^{3.3}]=0$ in $\CH^*(\Vcal\smallsetminus\Zcal)/([\Dcal^3],\zeta)$.
\end{prop}

\begin{proof}
First observe that the pullback morphism $\CH^*(\Vcal)\to\CH^*(\Vcal_1)$ is injective, hence to compute $[\Dcal^{3,3}]$ we can equivalently compute its pullback $[\Dcal^{3,3}_1]$ in $\CH^*(\Vcal_1)$. By Proposition \ref{prop:Vcal1 quotient}, this ring is isomorphic to the equivariant Chow ring $\CH_{\Gamma}^*(V_1)$.

Recall that $\Gamma$ is the quotient of $\overline{\Gamma}:=\Gm^{3}\rtimes \C_2$ by $\mmu_{2}=\langle (-1,-1,-1,\id)\rangle$. We also observed that $\Gamma\simeq \Gm^{ 3}\rtimes \C_2$, and that the quotient homomorphism is given by
\[ (u,v,t,\sigma) \in \overline{\Gamma} \longmapsto (u^{-1}t,v^{-1}t,uv,\sigma)\in \Gamma.  \]

The representation $V_1$ is the vector space of pairs of cubic binary forms $(f,g)$ such that $f(0,1)=g(0,1)$. The action of $\overline{\Gamma}$ on $V_1$ is
\[ (u,v,t,\id)\cdot (f(x_0,x_1),g(x_0,x_1)) = (t^2f (u^{-2}x_0,x_1),t^2g(v^{-2}x_0,x_1)), \]
and $\sigma$ exchanges $f$ and $g$. If $(a,b,d,\id)$ indicates an element of $\Gamma$, then $u^2=a^{-1}bd$, $v^2=ab^{-1}d$ and $t^2=abd$. From this the action of $\Gamma$ on $V_1$ can be easily reconstructed.

By construction, the stack $\Dcal^{3,3}_1$ is the quotient by $\Gamma$ of the closed subscheme of $V_1$ whose points are the pairs $(f^3,g^3)$ with $f(0,1)=g(0,1)$.

Consider the injective homomorphisms of groups
\begin{align*}
    \Gm^{ 2}\times \Gm \hookrightarrow \Gamma,&\quad (a,b,d)\longmapsto (a,b,d,\id), \\
    \mmu_{2}\times\C_2\times\Gm \hookrightarrow \Gamma,&\quad (-1,\sigma,d)\longmapsto (-1,-1,d,\sigma).
\end{align*}
These maps induce by restriction a ring homomorphism
\[ \CH^*_{\Gamma}(V_1) \longmapsto \CH^*_{\Gm^{ 2}\times \Gm}(V_1)\times \CH^*_{\mmu_{2}\times\C_2\times\Gm}(V_1) \]
which is injective because of \cite{VisPGLp}*{Proposition 6.4.(d)}. This reduces the computation of $[\Dcal^{3,3}_1]$ to the computation of its restrictions to the two equivariant Chow rings on the right.

First we compute $[\Dcal^{3,3}_1]$ in $\CH^*_{\mmu_{2}\times\C_2\times\Gm}(V_1)$.
Define $E_{\eta}$, $E_{\xi}$ and $E_{\delta}$ as the rank one representations determined respectively by the characters
\begin{align*}
     \chi_\eta : \mmu_{2}\times\C_2\times\Gm \longrightarrow \mmu_{2} \\
     \chi_\xi : \mmu_{2}\times\C_2\times\Gm \longrightarrow \C_2  \\
     \chi_\delta : \mmu_{2}\times\C_2\times\Gm \longrightarrow \Gm . 
\end{align*}
Then the first Chern classes of these representations, which we denote $\eta$, $\xi$ and $\delta$, generate $\CH^*_{\mmu_{2}\times\C_2\times\Gm}(V_1)$.

The subgroup $\mmu_{2}$ acts trivially on $V_1$, hence we only need to take into account the action of $\C_2\times\Gm$. Let $E$ be the rank two representation of $\Gm$ of weights $(\delta, 0)$. Then we have an equivariant decomposition
\[\left(\Sym^3E_{(\delta,0)}^{\vee}\right)\oplus \left(\Sym^3E_{(\delta,0)}^{\vee}\otimes {\rm sgn}\right) \xrightarrow{\simeq} V_1,\quad (f,g)\longmapsto (f+g,f-g) \]
where ${\rm sgn}$ is the sign representation of $\C_2$. In other terms, the subgroup $\C_2$ acts trivially on the first factor and as multiplication by $-1$ on the other factor.

Let $Z_3\subset \Sym^3 E_{(\delta,0)}^{\vee} $ be the invariant subscheme of forms with a triple root. Then:
\[ [\Dcal_1^{3,3}]_{\C_2\times\Gm} = [Z_3]_{\Gm}\cdot [Z_3]_{\C_2\times\Gm}.\]
We claim that $[Z_3]_{\Gm}$ is zero in $\CH^2(\Sym^3 E_{(\delta,0)}^{\vee} )$: this will imply that $[\Dcal_1^{3,3}]$ is zero in $\CH^*_{\mmu_{2}\times\C_2\times\Gm}(V_1)$.

For this, we use once again the trick of passing to the projectivization. Call $h$ the hyperplane class of $\PP(\Sym^3 E_{(\delta,0)}^{\vee})$.
Observe that $\Sym^3 E_{(\delta,0)}^{\vee}\smallsetminus\{0\}$ is the $\Gm$-torsor associated to the equivariant line bundle $\Ocal(-1)\otimes E_{\delta}$. By \cite{VisPGLp}*{Lemma 4.1} the kernel of the pullback morphism
\[\CH^*_{\mmu_{2}\times\C_2\times\Gm}(\PP(\Sym^3 E_{(\delta,0)}^{\vee}))\longrightarrow \CH^*_{\mmu_{2}\times\C_2\times\Gm}(\Sym^3 E_{(\delta,0)}^{\vee}\smallsetminus\{0\}) \]
is then equal to the equivariant first Chern class of $\Ocal(-1)\otimes E_{\delta}$, which is $\delta-h$. In other terms, the class $[Z_3]$ is equal to the class of its projectivization $[\overline{Z}_3]$ after substituting $h$ with $\delta$.

We can explicitly compute $[\overline{Z}_3]$ by means of Proposition \ref{prop:formulas}. We obtain:
\[ [\overline{Z}_3]=3(h-2\delta)(h-\delta), \]
hence after substituting $h$ with $\delta$ we get zero, as claimed.

Next we compute the class of $[\Dcal^{3,3}_1]$ in $\CH^*_{\Gm^{2}\times\Gm}(V_1)$. Let $\alpha$, $\beta$ and $\delta$ be the generators of this Chow ring, corresponding to the characters $u^{-1}t$, $v^{-1}t$ and $uv$.

Set $\gamma_1=\beta+\delta-\alpha$ and $\gamma_2=\alpha+\delta-\beta$, so that they correspond respectively to the characters $u^2$ and $v^2$. Define the representation
\[ U_1:=\left(\Sym^3 E_{\gamma_1,0}^{\vee}\oplus\Sym^3 E_{\gamma_2,0}^{\vee}\right)\otimes E_{\alpha+\beta+\delta}. \]
Then $V_1$ is a subrepresentation of $U_1$, and $[\Dcal^{3,3}_1]$ is the pullback to $V_1$ of the equivariant cycle class of $Y$, the subscheme of pairs of the form $(f^3,g^3)$. As the pullback induces an isomorphism at the level of Chow rings, it is completely equivalent to compute $[Y]$ instead of $[\Dcal^{3,3}_1]$.

To compute the class of $Y$, consider its image $\overline{Y}$ in $\PP(\Sym^3E_{\gamma_1,0})\times \PP(\Sym^3E_{\gamma_2,0})$. The pullback 
\[ \CH^*_{\Gm^{3}}(\PP(\Sym^3E^{\vee}_{\gamma_1,0})\times \PP(\Sym^3E^{\vee}_{\gamma_2,0}))\longrightarrow \CH^*_{\Gm^{3}}(U_1\smallsetminus\{0\}) \]
is surjective with kernel generated by $h_1-(\alpha+\beta+\delta)$ and $h_2-(\alpha+\beta+\delta)$, by the same argument used in the earlier case.
Therefore, to compute $[Y]$ is enough to compute $[\overline{Y}]$ and then substitute $\alpha+\beta+\delta$ to both $h_1$ and $h_2$.

The class of $\overline{Y}$ is equal to $(i\times i )_*1$, where
\[ i\times i : \PP(E^{\vee}_{(\gamma_1,0)})\times \PP(E^{\vee}_{(\gamma_2,0)}) \longrightarrow \PP(\Sym^3E^{\vee}_{\gamma_1,0})\times \PP(\Sym^3E_{\gamma_2,0}),\quad (f,g)\mapsto (f^3,g^3).\]
It is easy to check that $(i\times i )_*1=i_*1\cdot i_*1$, which is equal by Proposition \ref{prop:formulas} to
\[ 9(h_1-2\gamma_1)(h_1-\gamma_2)(h_2-2\gamma_2)(h_2-\gamma_2). \]
We substitute $\gamma_1$ with $\beta+\delta-\alpha$, $\gamma_2$ with $\alpha+\delta-\beta$ and both $h_1$ and $h_2$ with $\alpha+\beta+\delta$, and we get
\[ [\Dcal^{3,3}_1]=[Y]=36\alpha\beta(\delta^2-2(\alpha+\beta)\delta + 16\alpha\beta -3(\alpha+\beta)^2) \]
in $\CH^*_{\Gm^{3}}(V_1)$.

The injectivity of the ring homomorphism 
\[\CH^*_{\Gamma}(V_1) \longmapsto \CH^*_{\Gm^{2}\times \Gm}(V_1)\times \CH^*_{\mmu_{2}\times\C_2\times\Gm}(V_1) \]
together with the fact that the restriction of $[\Dcal^{3,3}_1]$ to $\CH^*_{\mmu_{2}\times\C_2\times\Gm}(V_1)$ is zero implies that we only have to substitute $\lambda_1$ to $\alpha+\beta$ and $\lambda_2$ to $\alpha\beta$ in the expression above to finally compute $[\Dcal^{3,3}_1]$ in $\CH^*_{\Gamma}(V_1)$. We obtain:
\[ [\Dcal^{3,3}_1] = 36\lambda_2(\delta^2-2\lambda_1\delta + 16\lambda_2 -3\lambda_1^2). \]

We easily deduce from the injectivity of the ring homomorphism $\CH^*(\Vcal)\to\CH^*(\Vcal_1)$ that the same expression holds for $[\Dcal^{3,3}]$ in $\CH^*(\Vcal)$. A straightforward computation shows that $[\Dcal^{3,3}]$ belongs to the ideal $([\Zcal],[\Dcal^{3,3}],\zeta)$, and concludes the proof.
\end{proof}
Putting all together, we obtain our main result.
\begin{thm}\label{thm:Chow M2bar}
    Suppose ${\rm char}(k)=0$. Then:
    \[ \CH^*(\overline{\Mcal}_2)\simeq \ZZ[\lambda_1,\lambda_2,\delta_1]/(2\delta_1^2+2\lambda_1\delta_1, \delta_1^3+\delta_1^2\lambda_1, 24\lambda_1^2-48\lambda_2,20\lambda_1\lambda_2-4\delta_1\lambda_2) \]
\end{thm}

\subsection{Extension to positive characteristic}
In this last Subsection we show how to extend the computations of the previous part to positive characteristic.

Let $R$ be a discrete valuation ring, with fraction field $K$ of characteristic zero and residue field $k$ of characteristic $p>3$. All the stacks defined so far can actually be defined over $\Spec(R)$. Moreover, the relative intersection theory developed in \cite{Ful}*{\S 20.3} can be easily extended so to apply also to quotient stacks, by using equivariant approximations as in \cite{EG}.

Therefore, by \cite{Ful}*{Corollary 20.3} there is a well defined specialization morphism
\[  \CH^*(\Vcal_K\smallsetminus\Zcal_K) \xrightarrow{\rm sp} \CH^*(\Vcal_k\smallsetminus\Zcal_k) \]
which is a ring homomorphism.

Our computation of $\CH^*(\Vcal\smallsetminus\Zcal)$ is independent of the characteristic, and it shows that the top horizontal arrow in the diagram above is an isomorphism of rings: by construction, it sends $\lambda_{i,K}$ to $\lambda_{i,k}$ and $\delta_{1,K}$ to $\delta_{1,k}$. Moreover, it also sends $[\Dcal^3_{K}]$ to $[\Dcal^3_{k}]$, $\zeta_K$ to $\zeta_k$ and $[\Dcal_K^{3,3}]$ to $[\Dcal^{3,3}_k]$. This immediately implies that
\[ ([\Dcal^3_k],\zeta_k,[\Dcal^{3,3}_k]) = (24\lambda_1^2-48\lambda_2,20\lambda_1\lambda_2-4\delta_1\lambda_2)\]
in $\CH^*(\Vcal_k\smallsetminus\Zcal_k)$.

Now we apply Corollary \ref{cor:Chow M2bar abs}, which holds over every field of characteristic $\neq 2,3$ to extend the computation of Theorem \ref{thm:Chow M2bar} over more general fields.
\begin{thm}\label{thm:Chow M2bar ii}
    Suppose that the ground field has characteristic $\neq 2, 3$. Then:
    \[ \CH^*(\overline{\Mcal}_2)\simeq \ZZ[\lambda_1,\lambda_2,\delta_1]/(2\delta_1^2+2\lambda_1\delta_1, \delta_1^3+\delta_1^2\lambda_1, 24\lambda_1^2-48\lambda_2,20\lambda_1\lambda_2-4\delta_1\lambda_2) \]
\end{thm}

\appendix
\section{Deformation theory of nodal singularities}\label{sec:app}
The following is well known to the experts, but we do not know a fully rigorous treatment in the literature.

Consider the stack $\fM_{g} \arr \aff$ of geometrically reduced and geometrically connected local complete intersection curves of arithmetic genus $g$ (from now on \emph{lci curves}); this is a smooth algebraic stack. A family of lci curves $\phi\colon C \arr S$, where $S$ is an algebraic stack of finite type, is called \emph{versal} if the corresponding morphism $S \arr \fM_{g}$ is smooth. If $\phi\colon C \arr S$ is versal, both $S$ and $C$ are smooth over $k$.

Let $\phi\colon C \arr S$ is a family of lci curves, denote by $\Delta \subseteq C$ its singular locus, defined, as usual, by the first Fitting ideal of the sheaf of relative differentials $\Omega_{C/S}$; this a closed subscheme of $C$, finite over $S$. Let $\Sigma \subseteq \Delta$ an open and closed subscheme of $\Delta$ such that the geometric points of $\Sigma \subseteq C$ correspond to nodes in the geometric fibers of $\phi\colon C \arr S$, and the projection $\Sigma \arr S$ is injective on geometric fibers.

We will use the following fundamental structure result (see for example \cite[\S 7]{TVdef}).

\begin{lm}\label{lem:local-around-node}
Assume that $k$ is algebraically closed and $S$ is a scheme. Let $p \in \Sigma(k)$ be a rational point; call $q$ its image in $S$. After passing to an \'{e}tale neighborhood of $q$ in $S$, we have that $S = \spec R$ is affine, and there is an \'{e}tale neighborhood $p \in U \xarr{\phi} C$ of $p$ in $C$, and an element $f \in R$ vanishing at $q$,  with an \'{e}tale map $\psi\colon U \arr \spec\bigr(R[x,y]/(xy-f)\bigr)$ such that $\phi^{-1}(\Sigma) = \psi^{-1}\spec\bigr(R[x,y]/(x,y,f)\bigr) \subseteq U$.

%The element $f \in R$ is uniquely determined, \'{e}tale locally around $q$, up to multiplication by a unit. It will be called \emph{a local equation for $\Sigma$ at $p$}.

Furthermore, if $\phi\colon C \arr S$ is versal, then $f$ is part of a system of parameters of $S$ at $q$.

%Furthermore, suppose that $p_{1}$, \dots,~$p_{r} \in U(k)$ map to the same point $q \in S(k)$, and let $f_{1}$, \dots,~$f_{r}$ be local equations for $\Sigma$ at $p$. If $\phi\colon C \arr S$ is versal, then $f_{1}$, \dots,~$f_{r}$ are parts of a system of parameters of $S$ at $q$. 
\end{lm}

From this we obtain the following.

\begin{prop}\label{prop:lci}
The composite $\Sigma \subseteq \phi\colon C \arr S$ is a closed embedding. Furthermore, if $\phi\colon C \arr S$ is versal, then $\Sigma$ is smooth, of codimension~$2$ in $C$, and codimension~$1$ in $S$.
\end{prop}

\begin{proof}
We can assume that $S$ is a scheme, and $k$ is algebraically closed. Then it follows from Proposition~\ref{lem:local-around-node}, with straightforward calculations, that $\Sigma \arr S$ is unramified, and that $\Sigma$ is smooth if $\phi\colon C \arr S$ is versal. Since $\Sigma \arr S$ is also finite and injective on geometric points, it follows that it is an embedding.
\end{proof}

Call $C_{\Sigma} \subseteq C$ the inverse image of $\Sigma \subseteq S$ in $C$; clearly $\Sigma$ is contained in $C_{\Sigma}$. The point of the following is that $\phi\colon C \arr S$ is versal, then we can describe both the normal bundle of $\Sigma \subseteq C$ and of $\Sigma\subseteq S$ using the embedding $\Sigma \arr C_{\Sigma}$.

Given a closed embedding $Y \subseteq X$ we will denote by $\rI_{Y}X$ the sheaf of ideals of $Y$ in $X$, by $\rI_{Y}^{(n)}X$ the sheaf $(\rI_{Y}X)^{n} \times_{\cO_{X}}\cO_{Y} = (\rI_{Y}X)^{n}/(\rI_{Y}X)^{n+1}$, by $\rN_{Y}X$ the relative spectrum over $Y$ of the symmetric algebra sheaf $\undersym_{\cO_{Y}}^{\smallbullet}\rI^{(1)}_{Y}X$, and by $\Gamma_{Y}X$ the normal cone of $Y$ in $X$. We have an obvious embedding of cones $\Gamma_{Y}X \subseteq \rN_{Y}X$.

If $f\colon X \arr X'$ is a morphism of schemes such that the composite $Y \arr X \xarr{f} X'$ is still a closed embedding, then for each $n \geq 0$ there are induced homomorphisms of sheaves of $\cO_{Y}$-modules $\rI_{Y}^{(n)}X' \arr \rI_{Y}^{(n)}X$, which induce morphisms of cones $\rN_{Y}X \arr \rN_{Y}X'$ and $\Gamma_{Y}X \arr \Gamma_{Y}X'$.

The following describes the normal bundle to $\Sigma$ in $C$ in term of the embedding $\Sigma \subseteq \Sigma$.

\begin{prop}\call{prop:normal-in-the-family}\hfil
\begin{enumerate1}

\itemref{1} The sheaf of $\cO_{\Sigma}$-modules $\rI^{(1)}_{\Sigma}C_{\Sigma}$ is locally free of rank~$2$, and the induced homomorphism $\rI^{(1)}_{\Sigma}C \arr \rI^{(1)}_{\Sigma}C_{\Sigma}$ is an isomorphism.

\itemref{2}  The cone $\rN_{\Sigma}C_{\Sigma} \arr \Sigma$ is a vector bundle of rank~$2$, and the morphism of cones $\rN_{\Sigma}C_{\Sigma} \arr \rN_{\Sigma}$ induced by the embedding $C_{\Sigma} \subseteq C$ is an isomorphism.
\end{enumerate1}
\end{prop}

\begin{proof}
For \refpart{prop:normal-in-the-family}{1}, we can assume that $S$ is a scheme, and $k$ is algebraically closed. It is enough to prove that $\rI^{(1)}_{\Sigma}C_{\Sigma}$ is locally free of rank~$2$ as an $\cO_{\Sigma}$-module, and that $\rI_{C_{\Sigma}}C$ is contained in $(\rI_{\Sigma}C)^{2}$. This follows easily from Proposition~\ref{lem:local-around-node}.

\refpart{prop:normal-in-the-family}{2} follows from \refpart{prop:normal-in-the-family}{1}.
\end{proof}

Set $N \eqdef \rN_{\Sigma}C_{\Sigma}$. There is a natural \'{e}tale representable double cover $\pi\colon \overline{\Sigma} \arr \Sigma$; a geometric point of $\overline{\Sigma}$ corresponds to a geometric point $p$ of $\Sigma$, together with one of the two branches of the fiber of $\phi\colon C \arr S$ at $p$. Formally, we define $\pi\colon \overline{\Sigma} \arr \Sigma$ as the projectivization $\PP(\Gamma_{\Sigma}C_{\Sigma})$ of the normal cone of $\Sigma$ in $C_{\Sigma}$. The fact that the projection $\overline{\Sigma} \arr \Sigma$ is an \'{e}tale double cover follows easily using Lemma~\ref{lem:local-around-node}. Denote by $\tau\colon \overline{\Sigma} \arr \overline{\Sigma}$ the involution switching the two sheets of $\pi\colon \overline{\Sigma} \arr \Sigma$. 

The embedding $\Gamma\subseteq N$ gives an embedding $\overline{\Sigma} \subseteq \PP(N)$; denote by $L \subseteq \pi^{*}N$ the pullback of the invertible sheaf $\cO_{\PP(N)}(-1)$. We also have another invertible subsheaf $\tau^{*}L \subseteq \tau^{*}\pi^{*}N = \pi^{*}N$, giving a decomposition $L \oplus \tau^{*}L = N$. Furthermore, the corresponding homomorphism $\pi_{*}L \arr N$ is an isomorphism.

There are two structures of $\generate{\tau}$-equivariant invertible sheaf on $L\otimes \tau^{*}L$, one in which acts by sending $t \times t'$ to $t'\otimes t$, and the second one to $-t'\otimes t$. These yield by descent two invertible sheaves on $\Sigma$; the second one gives the determinant $\det \pi_{*}L = \det N$, the first one an invertible sheaf that we denote by $M$ (this is the norm of $L$). An alternate description of $M$ is as $(\det N)\otimes P$, where $P$ is the invertible sheaf on $\Sigma$ obtained by descent from the representation $\generate{\tau}\arr\gm$ sending $\tau$ to $-1$.

Notice that the formation of $M$ commutes with base change under arbitrary maps $S' \arr S$.

This gives a description of the normal bundle of $\Sigma$ in $S$, when $\phi\colon C \arr S$ is versal.

\begin{prop}\label{prop:normal-in-the-base}
Assume that $\phi\colon C \arr S$ is versal. Then there is a canonical isomorphism $M \simeq \rN_{\Sigma}S$ of invertible sheaves on $\Sigma$. This isomorphism commutes with smooth base change $S' \arr S$.
\end{prop}

\begin{proof}
The fact that $\phi\colon C \arr S$ is singular along $\Sigma$ ensures that the homomorphism $\phi^{*}\rI_{\Sigma}S \arr \rI_{\Sigma}C$ has image contained in $(\rI_{\Sigma}C)^{2}$; by restricting to $\Sigma$ we obtain a homomorphism $\rI^{(1)}_{\Sigma}S \arr \rI^{(2)}_{\Sigma}C$. A straightforward local calculation using Lemma~\ref{lem:local-around-node} reveals that this is injective, so we get an embedding $\rI^{(1)}_{\Sigma}S \subseteq \rI^{(2)}_{\Sigma}C$.

On the other hand, let us produce a homomorphism $M^{\vee} \arr \rI^{(2)}_{\Sigma}C$. We have a decomposition $\pi^{*}(\rI^{(1)}_{\Sigma}C_{\Sigma}) = L^{\vee} \oplus \tau^{*}L^{\vee}$, which gives an embedding
   \[
   L^{\vee}\otimes\tau^{*}L^{\vee} \subseteq \pi^{*}(\rI^{(1)}_{\Sigma}C_{\Sigma})
   \otimes \pi^{*}(\rI^{(1)}_{\Sigma}C_{\Sigma})\,.
   \]
By composing this with the product $\pi^{*}(\rI^{(1)}_{\Sigma}C_{\Sigma})\otimes \pi^{*}(\rI^{(1)}_{\Sigma}C_{\Sigma}) \arr \pi^{*}(\rI^{(2)}_{\Sigma}C_{\Sigma})$ we obtain a homomorphism
   \[
   L^{\vee}\otimes\tau^{*}L^{\vee} \arr \pi^{*}(\rI^{(2)}_{\Sigma}C_{\Sigma})\,.
   \]
Because of the commutativity of multiplication this is $\generate{\tau}$-equivariant, so it descends to a homomorphism $M^{\vee} \arr \rI^{(2)}_{\Sigma}C_{\Sigma}$. Once again, local calculations using Lemma~\ref{lem:local-around-node} reveal that this is injective, its image coincides with the image of $\rI^{(1)}_{\Sigma}S$. Thus we obtain an isomorphism $M^{\vee} \simeq \rI^{(1)}_{\Sigma}S$, and by dualizing the desired isomorphism $M \simeq \rN_{\Sigma}S$.
\end{proof}

\end{document}